\documentclass[final,onefignum,onetabnum]{siamonline181217}

\makeatletter
\def\moverlay{\mathpalette\mov@rlay}
\def\mov@rlay#1#2{\leavevmode\vtop{%
   \baselineskip\z@skip \lineskiplimit-\maxdimen
   \ialign{\hfil$\m@th#1##$\hfil\cr#2\crcr}}}
\newcommand{\charfusion}[3][\mathord]{
    #1{\ifx#1\mathop\vphantom{#2}\fi 
        \mathpalette\mov@rlay{#2\cr#3}
      }
    \ifx#1\mathop\expandafter\displaylimits\fi}
\makeatother

\newcommand{\cupdot}{\charfusion[\mathbin]{\cup}{\cdot}}

\usepackage{amsfonts}

\usepackage{array}

\usepackage[caption=false]{subfig}
\usepackage{pgfplots}
\usepackage{overpic}
\newsiamthm{claim}{Claim}
\newsiamremark{remark}{Remark}
\newsiamremark{hypothesis}{Hypothesis}
\crefname{hypothesis}{Hypothesis}{Hypotheses}

\usepackage{algorithmic}

\usepackage{graphicx}
\usepackage{amsmath,amssymb}
\usepackage{graphicx}
\usepackage{multicol}
\usepackage[utf8]{inputenc}
\usepackage[english]{babel}
\usepackage{amsopn}

\usepackage[utf8]{inputenc}
\usepackage{tikz}
\usetikzlibrary{shapes.geometric, arrows}

\usepackage{graphicx}
\usepackage{xspace}
\usepackage{bold-extra}
\usepackage[most]{tcolorbox}

\colorlet{texcscolor}{blue!50!black}
\colorlet{texemcolor}{red!70!black}
\colorlet{texpreamble}{red!70!black}
\colorlet{codebackground}{black!25!white!25}

\lstdefinestyle{siamlatex}{%
  style=tcblatex,
  texcsstyle=*\color{texcscolor},
  texcsstyle=[2]\color{texemcolor},
  keywordstyle=[2]\color{texemcolor},
  moretexcs={cref,Cref,maketitle,mathcal,text,headers,email,url},
}

\tcbset{%
  colframe=black!75!white!75,
  coltitle=white,
  colback=codebackground, 
  colbacklower=white, 
  fonttitle=\bfseries,
  arc=0pt,outer arc=0pt,
  top=1pt,bottom=1pt,left=1mm,right=1mm,middle=1mm,boxsep=1mm,
  leftrule=0.3mm,rightrule=0.3mm,toprule=0.3mm,bottomrule=0.3mm,
  listing options={style=siamlatex}
}

\newtcblisting[use counter=example]{example}[2][]{%
  title={Example~\thetcbcounter: #2},#1}

\newtcbinputlisting[use counter=example]{\examplefile}[3][]{%
  title={Example~\thetcbcounter: #2},listing file={#3},#1}

\DeclareTotalTCBox{\code}{ v O{} }
{ 
  fontupper=\ttfamily\color{black},
  nobeforeafter,
  tcbox raise base,
  colback=codebackground,colframe=white,
  top=0pt,bottom=0pt,left=0mm,right=0mm,
  leftrule=0pt,rightrule=0pt,toprule=0mm,bottomrule=0mm,
  boxsep=0.5mm,
  #2}{#1}

\theoremstyle{definition}

\begin{tcbverbatimwrite}{tmp_\jobname_header.tex}
\title{Variational Path Optimization of Linear Pentapods  with a Simple Singularity Variety\thanks{Submitted to the editors DATE.
\funding{The first author was funded by the Doctoral College ``Computational Design" of TU Wien and the second author is supported by Grant No.~P~30855-N32 of the Austrian Science Fund FWF.}}}

\author{Arvin Rasoulzadeh\thanks{Center for Geometry and Computational Design,
TU Wien.\newline\ (\email{rasoulzadeh@geometrie.tuwien.ac.at}).}
\and Georg Nawratil\thanks{(\email{nawratil@geometrie.tuwien.ac.at}).}}

\headers{Variational Path Optimization of linear Pentapods}{Arvin Rasoulzadeh and Georg Nawratil}
\end{tcbverbatimwrite}
\input{tmp_\jobname_header.tex}

\begin{document}
\maketitle

\begin{tcbverbatimwrite}{tmp_\jobname_abstract.tex}
\begin{abstract}
The class of \emph{linear pentapod with a simple singularity variety} is obtained by imposing architectural restrictions on the design in such a way that the manipulator's singularity variety is linear in orientation/position variables. It turns out that such a simplification leads to crucial computational advantages while maintaining the machine's applications in some fundamental industrial tasks such as 5-axis milling and laser cutting. 
We assume  that a path between a given start- and end-pose of the end-effector is known, which is singularity-free and within the manipulator's workspace. 
An optimization process of this initial path is proposed in such a way that the parallel robot increases its distance to the singularity loci  
while the motion is being smoothed.
In our case the computation time of the optimization is improved as we are dealing with pentapods having simple singularity varieties allowing a closed form solution for the local extrema of the singularity-distance function. Formally, this process is called \emph{variational path optimization} which is the systematic optimization of a path by manipulating its \emph{variations of energy} and \emph{distance to the obstacle}, which in this case is the singularity variety. In this process some physical limits of the mechanical joints are also taken into account.
\end{abstract}


\end{tcbverbatimwrite}
\input{tmp_\jobname_abstract.tex}
\begin{tabular}{ |p{0.7\columnwidth} p{0.2\textwidth}|}
 \hline
 \multicolumn{2}{|c|}{Nomenclature} \\
 \hline
 Name                                                                       & Symbol\\
 \hline
 Real numbers                                                               & $\mathbb{R}$\\
 Algebraic variety                                                          & $\mathbf{V}$\\
 Generalized cylinder                                                       & $\Gamma$\\
 Pentapod's singularity variety                                             & $\Sigma$\\
 Cost function                                                              & $\mathcal{C}$\\
 Objective function                                                         & $\mathcal{C}^{\prime}$\\
 Initial guess i-th coordinate                                              & $p_{i}$\\
 Pose i-th coordinate                                                       & $u_{i}$\\
 pedal points with respect to the breakpoint $p^i$ on $\Sigma$/$Q$-variety  & $q^{i}$/$\mathfrak{q}^{i}$\\
 Object-oriented metric                                                     & $\mathfrak{d}$\\
 Orthogonal projection with respect to object-oriented metric               & $\pi$\\
 Orthogonal projection with respect to Euclidean metric                     & $\mathrm{Pr}$\\
 Number of breakpoints                                                      & $n$\\
 Number of pedal points with respect to the breakpoint $p$             & $m_p$\\             
 \hline
\end{tabular}
\section{Introduction}
\label{sec:intro}
A \emph{linear pentapod} is a five degree-of-freedom parallel robot consisting of five identical spherical-prismatic-spherical legs, where five platform anchor points are aligned along the motion platform $\ell$. Note that the prismatic joints are active while the spherical ones are passive. Within the paper at hand we only meet this type of manipulators possessing a planar base; i.e.\  the base anchor points are located on a plane (cf.\ Fig.\,\ref{fig:general}). The pose of the line $\ell$ is uniquely characterized by a position vector $\overrightarrow{p}\in\mathbb{R}^3$ and an orientation unit-vector $\overrightarrow{i}\in\mathbb{R}^3$ along $\ell$. Additionally the platform anchor points' coordinate vectors are defined by the equation ${m}_{i}=\overrightarrow{p}+r_{i}\overrightarrow{i}$ while the coplanar base anchor points are free to acquire ${M}_{i}=(x_{i},y_{i},0)$ coordinates for $i=1,...,5$.

Linear pentapods are capable of performing some fundamental industrial tasks such as 5-axis milling and cutting. Few years ago, the company \emph{Metrom} developed such a machine which, with the aid of a rotary table, is capable of full 5-sided machining \cite{weck2002parallel}.

Kinematic singularities have an important role to play in path-planning. These singular configurations are critical poses occurred by gaining some uncontrollable degree-of-freedom or the loss of stiffness in certain directions. The actuator forces may also become very large and cause the breakdown of the system. 
Therefore it is of crucial importance in the path-planning process not only to avoid singularities \emph{but also their neighborhoods}. 
In this paper we are only dealing with linear pentapods possessing simple singularity varieties (i.e.\
they are linear in position/orientation variables), which were determined by the authors in \cite{rasoulzadeh2019linear}. 

Under the assumption that a singularity-free path between a given start- and end-pose of the end-effector is known within the manipulator's workspace, our goal is to 
reshape this initial path iteratively by increasing its distance to the singularity loci while smoothing the path. 
This so-called \emph{variational path optimization} does not only take the \emph{variations of energy} and distance to the singularity variety 
(with respect to a reasonable metric) into account, but also some physical limits of the mechanical joints.

The paper is organized as follows: First, a review on the previous works is given in Section \ref{sec:review}. Since the materials used here are heavily dependent on \cite{rasoulzadeh2019linear, rasoulzadeh2019linear2}, most of the review is associated with the findings regarding the class of  linear pentapods with simple singularity varieties and the computation of singularity-free balls. 
In the next step, the required \emph{algebraic geometric} and \emph{differential geometric} settings are given in Section \ref{sec:algebraicgeometricsetup} and \ref{sec:differentialgeometricsetup}, respectively. 
Section\,\ref{sec:joint} presents the optimization of the initial curve under some physical restrictions, caused by joint limits. After that the algorithmic structure of the variational path optimization is discussed in Section \ref{sec:algorithmdetails} and its results are presented in Section\,\ref{sec:results}.
In addition, a flowchart is provided in the Appendix by which the reader can track the detailed explanation of each section of the 
\emph{variational path optimization} algorithm.
\begin{figure}[t!] 
\begin{center}   
   \begin{overpic}[width=110mm]{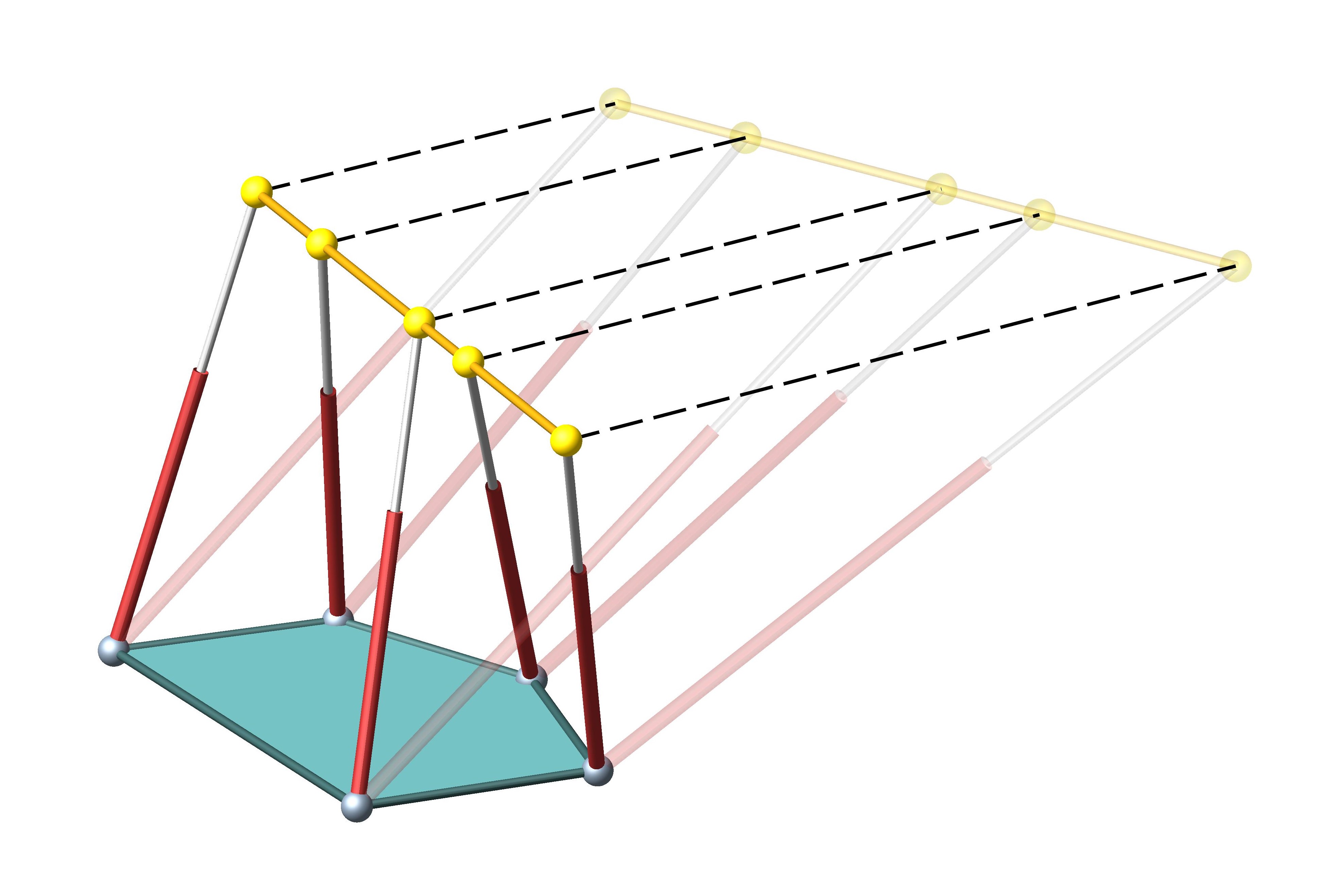}
	\begin{small}
   \put(22,49){\color{black}\line(-1,0){15}}
   \put(29,43){\color{black}\line(-1,0){22}}
   \put(33,40){\color{black}\line(-1,0){26}}
   \put(42,16){\color{black}\line(1,0){24}}
   \put(29,44.75){\color{black}\line(0,1){17}}
   \put(65,54.25){\color{black}\line(0,1){8}}
	\put(28.5,63){$\ell$}
   \put(64.5,63){$\ell^{\prime}$}
   \put(2,15){$M_1$}
   \put(12,53){$m_1$}
   \put(19,22){$M_2$}
   \put(2,49){$m_2$}
   \put(23,2){$M_3$}
   \put(2,43){$m_3$}
   \put(67,15){$M_4$}
   \put(2,40){$m_4$}
   \put(44,5){$M_5$}
   \put(42,37){$m_5$}
   \put(46.5,61.5){${m}_1^{\prime}$}
   \put(56,59){${m}_2^{\prime}$}
   \put(71,55){${m}_3^{\prime}$}
   \put(78,53){${m}_4^{\prime}$}
   \put(93,49){${m}_5^{\prime}$}
	\end{small}     
  \end{overpic} 
	\caption{ Illustration of a linear pentapod with a planar base. From a geometrical point of view, the value of Eq.\,\ref{distance}, for the two poses of the planar pentapod, is equal to $1/5$ of the sum of the squared lengths of the dashed lines between two platforms configurations.}
	\label{fig:general}
\end{center}
\end{figure}
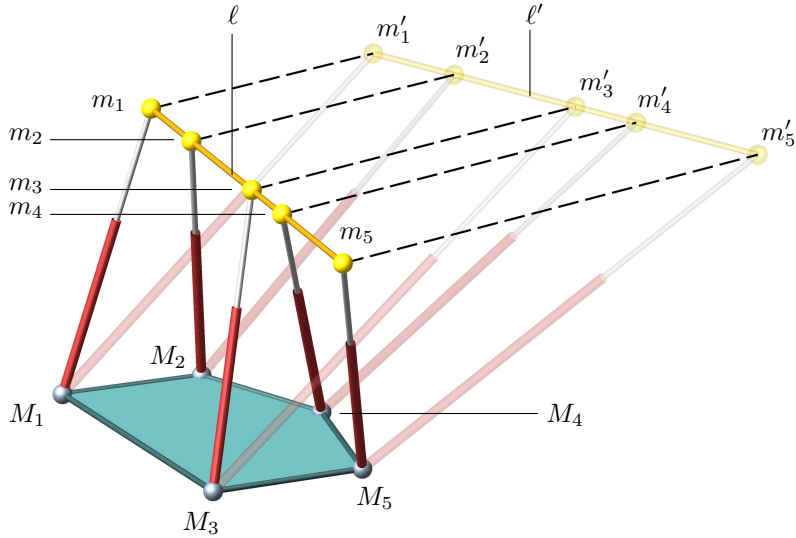 
\subsection{Review}\label{sec:review}
From the line-geometric point of view (cf.\ \cite{merlet1989singular}) a linear pentapod is in a singular configuration if and only if 
the five carrier lines of the legs belong to a linear line congruence \cite{pottmann2009computational}; i.e. the Pl\"ucker coordinates of 
these lines are linearly dependent. From this latter characterization the following algebraic one can be obtained (cf.\ 
\cite{rasoulzadeh2018rational}):\\
There exists a bijection between the configuration space of a linear pentapod and all points 
$(u_1,u_2,u_3,u_4,u_5,u_6)\in\mathbb{R}^{6}$ located on the singular quadric $\Gamma: {u_1}^2+{u_2}^2+{u_3}^2=1$, where $\overrightarrow{i} = ({u_1},{u_2},{u_3})$ 
determines the orientation of the linear platform $\ell$ and $\overrightarrow{p} = (u_4,u_5,u_6)$ its position. 
Then the set of all singular robot configurations is obtained as the intersection of $\Gamma$ with a \emph{cubic} 5-dimensional variety $\Sigma$ of $\mathbb{R}^{6}$, 
which can be written as $\Sigma:\, det\left({S}\right)=0$ with 
\begin{eqnarray}
{S}=\left( \begin {array}{ccccccc} 1&u_1&u_2&u_3&u_4&u_5&u_6
\\ 0&u_4&u_5&u_6&0&0&0
\\ 0&0&0&0&u_1&u_2&u_3\\ r_{2}&x_{2}
&y_{2}& 0 &r_{2}x_{2}&r_{2}y_{2}&0
\\ r_{3}&x_{3}&y_{3}& 0 &r_{3}x_{3}&r_
{3}y_{3}& 0 \\ r_{4}&x_{4}&y_{4
}& 0 &r_{4}x_{4}&r_{4}y_{4}& 0
	\\ r_{5}&x_{5}&y_{5}& 0 &r_{5}x_{5}&r_
{5}y_{5}& 0 \end {array} \right), 
\label{BorrasMatrix}
\end{eqnarray}
(according to \cite{borras2010singularity})
under the assumption that $x_1=y_1=z_1=r_1=0$. Note that this assumption can always be considered without loss of generality as the fixed/moving frame can be chosen in a way that the first base/platform anchor point is its origin. Furthermore, to relax the computations even more, the authors proved that it is possible to assume $M_{1}=(0,0,0)$, $M_{2}=(x_{2},0,0)$ and $M_{3}=(x_{3},y_{3},0)$ where $x_{2}y_{3}\neq 0$ and $r_{2}=1$ hold \cite{rasoulzadeh2019linear}.\\
\begin{figure}[t!] 
\begin{center}   
  \begin{overpic}[height=35mm]{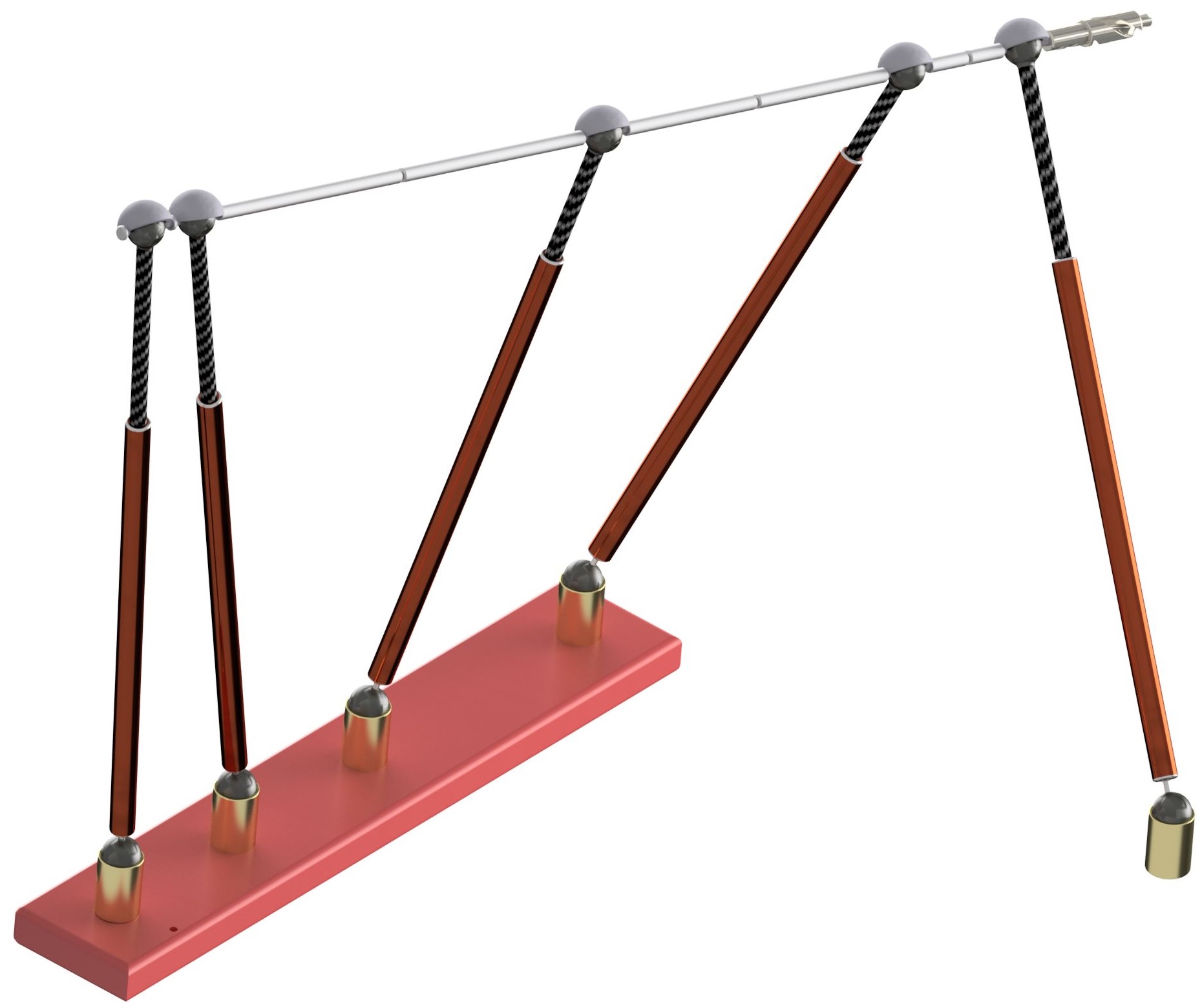}     
  \end{overpic}
   \hfill
  \begin{overpic}[height=36mm]{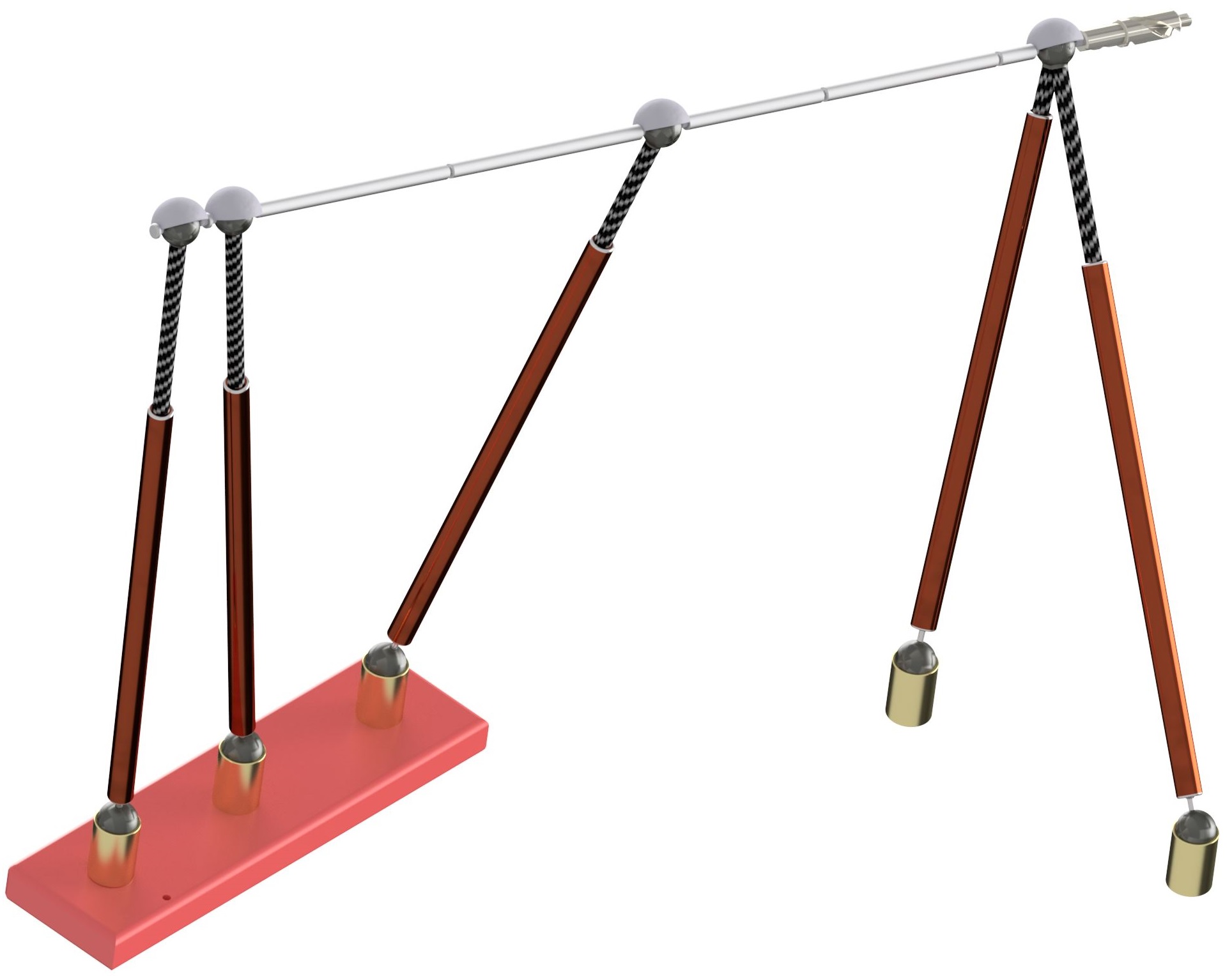}
	\begin{small}
	\end{small}     
  \end{overpic}
	\hfill
	\begin{overpic}[height=36mm]{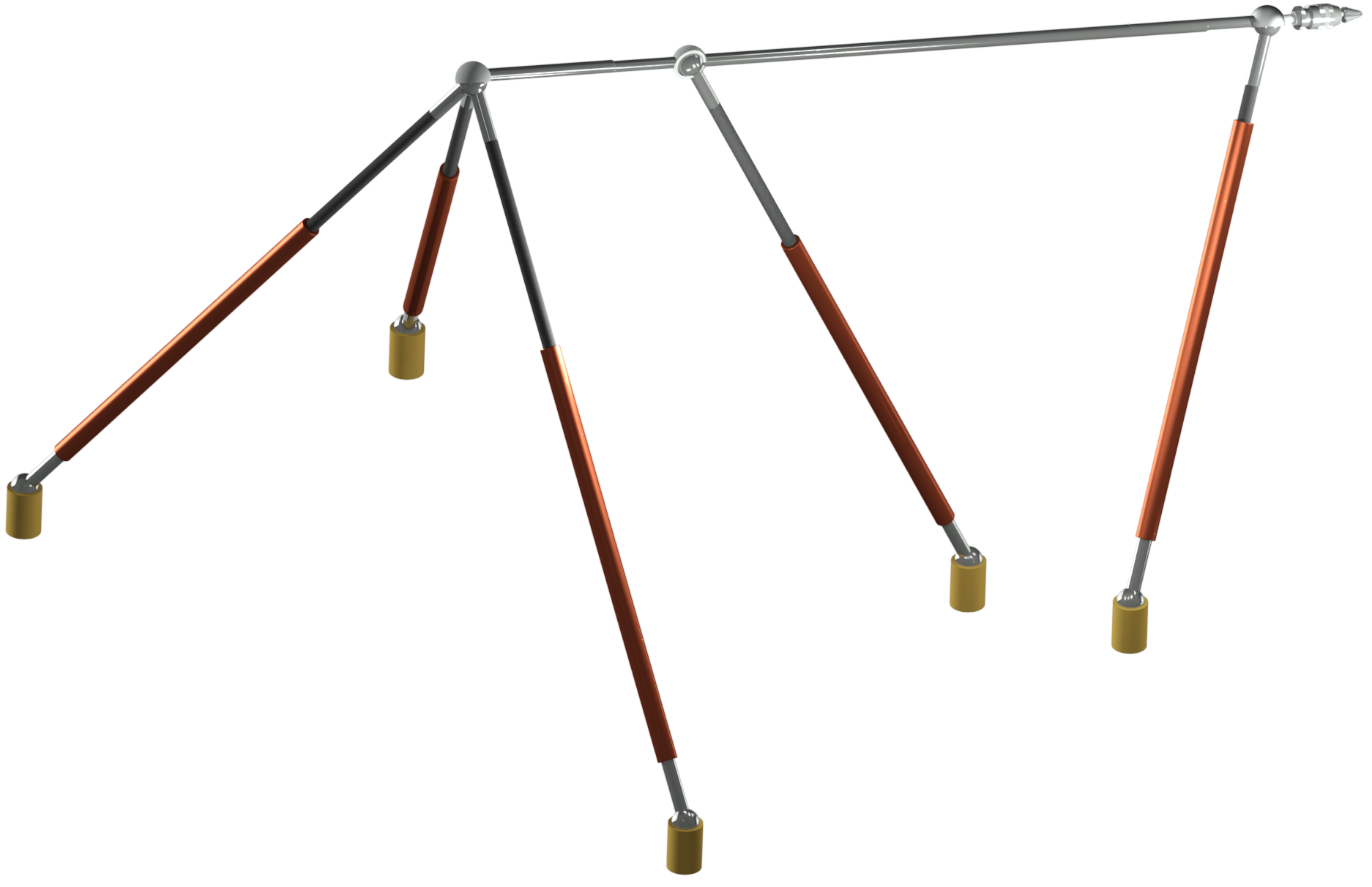}
	\begin{small}
	\end{small}     
  \end{overpic}
	\caption{Illustrations of LO-case pentapods. The geometric characterization of LO-cases is as follows:
1st-LO: $M_2$, $M_3$, $M_4$, $M_5$ are collinear (left), 
2nd-LO: $m_1=m_i$ and  $M_j$, $M_k$, $M_l$ are collinear with pairwise distinct $i,j,k,l\in\left\{2,3,4,5\right\}$ (middle),
3rd-LO: $m_1=m_i=m_j$ with pairwise distinct $i,j\in\left\{2,3,4,5\right\}$ (right).
}
	\label{fig:LO}
\end{center}
\end{figure}   
In \cite{rasoulzadeh2019linear}, the authors demonstrated the necessary and sufficient conditions for a linear pentapod to be linear in orientation/position variables. The linearity in orientation variables, which from now on is named \emph{LO-property} or \emph{LO-case}, 
is fulfilled for the following three types of linear pentapods which are depicted in Fig.\,\ref{fig:LO}. In fact, it is revealed that one can obtain all singular configurations of pentapods with LO-property in using one of the three classes, 
as they are equivalent with respect to $\Delta$-transforms mentioned in \cite{borras2010singularity1}.
For the linearity in position variables, the so-called \emph{LP-property} or \emph{LP-case}, 
authors demonstrated that a pentapod possesses this property iff there exists a \emph{singular affine map}\\
\begin{equation}
\kappa:\,\,(x_{i},y_{i}) \longmapsto r_{i}=\alpha x_{i} + \beta y_{i},
\end{equation}
with $x_{2} =\tfrac{1}{\alpha}$ and $i=1,\ldots , 5$.
\begin{definition}
A linear pentapod possessing a singularity variety linear in orientation/position variables is called a ``simple pentapod".
\label{def:simplepentapod}
\end{definition}  
It is noteworthy to emphasis that due to findings in \cite{rasoulzadeh2019linear}, simple pentapods must be with a planar base.\\
In \cite{rasoulzadeh2019linear2}, the computation of the maximum singularity-free balls for linear pentapods 
is thoroughly investigated. 
As the configuration space of the end-effector equals the space of oriented line-elements, 
we can adopt an object-oriented metric discussed in \cite{nawratil2017point} for our 
mechanical device as follows:
\begin{equation}\label{distance}
\mathfrak{d}(\ell, \mathcal{\ell^{'}})^{2}:=
\frac{1}{5}\sum_{j=1}^5{\|{m}_j-{m}^{'}_{j}\|}^{2},
\end{equation}
where $\ell$ and $\ell^{\prime}$ are two configurations and ${m}_j$ and ${m}^{'}_{j}$ denote the coordinate vectors of the corresponding platform anchor points (cf. Fig.\,\ref{fig:general}).\\
From \cite{rasoulzadeh2018rational} it was already known that the determination of the local extrema of the singularity-distance function is a polynomial problem of degree 80, which 
can be relaxed to a problem of degree 28 by expanding the transformation group from Euclidean to equiform motions by omitting the normalizing condition $\Gamma$. As the obtained distance of 
the relaxed problem is less or equal to the distance of the original problem, it can be used as the radius of a guaranteed singularity-free ball. 
Moreover in \cite{rasoulzadeh2019linear2} it was shown for the class of simple pentapods that the algebraic degree of this distance computation problem drops from 10  
to 3 (cf.\ Table\,\ref{table:pedals}). The resulting closed form solution offers interesting new concepts and strategies concerning path optimization and singularity avoidance. 
\begin{table}[t]
\centering
\begin{small}
{
 \begin{tabular}{||c |c| c| c||} 
 \hline 
                                               & Generic case  &  LP case  &  LO case \\ [0.5ex]
 \hline\hline
 Object-oriented metric case                                  &\color{red}80\color{black} &\color{red}10\color{black} &\color{red}10\color{black} \\[0.5ex]
 \hline
 Object-oriented metric case without normalizing condition    &\color{red}28\color{black} &\color{green}3\color{black} &\color{green}3\color{black}        \\ 
 \hline
\end{tabular}} \\[1ex]
\end{small}
\vspace{0.5	mm}
\caption{\color{black} Generic number of pedal points under different metric conditions: “Generic case" refers to the general pentapod (not necessarily with a simple singularity variety). Note that they are computed over the field of complex numbers and hence the real solutions might be lower. The red coloured cells indicate that these numbers are just experimental while the green ones are mathematically proven (cf. Theorem \ref{theorem:pedals}).}
\label{table:pedals}
\end{table}
\section{Algebraic Geometric Setup}
\label{sec:algebraicgeometricsetup}
Before plunging into the detailed optimization of the singularity-free paths, some remarks and definitions are of necessity to wrap up the mathematical structures of the pentapod's singularity locus. 
\subsection{$\Sigma$ - Variety}
\begin{definition}
The locus of the singularity polynomial of a ``simple pentapod" is called a ``$\Sigma$-variety".
\label{def:simple variety}
\end{definition}		
Now, tidying up the known information through definitions makes it possible to take a look at the structure of the $\Sigma$-varieties of both LO and LP-cases. 
\begin{lemma}
\label{lem:LO}
The $\Sigma$-variety of the LO-cases is the zero set of the following polynomial:
\begin{equation}
\Sigma : u_6\ \left[ u_6(\alpha u_1 + \beta u_2) - u_3 (\alpha u_4 + \beta u_5 -1) \right]=0,
\label{LO:sing} 
\end{equation} 
where $\alpha$ and $\beta$ are real numbers and $\alpha^2 + \beta^2 \neq 0$ (cf. \cite{rasoulzadeh2019linear}).
\end{lemma}
\begin{proof}
It can be shown by a series of $\Delta$-transforms (cf. \cite{borras2010singularity1}), that the singularity loci of all three cases are identical. Now, by substituting the relations between architecture parameters of one of the cases (e.g. 3rd-case identified by pairwise distinct indices $i,j,k \in \{3,4,5\}$) into $\det\left({S}\right)$ (cf. Eq.\,\ref{BorrasMatrix}): 
\begin{equation}
x_{2}:=\frac{1}{\alpha},\ \ \ x_{i}:=\frac{1-\beta y_{i}}{\alpha},\ \ \ r_{j}=r_{k}:=0,
\end{equation}
one obtains Eq.\,\ref{LO:sing}. 
\end{proof}
\begin{lemma}
\label{lem:LP}
The $\Sigma$-variety of the LP-case is the zero set of the following polynomial:
\begin{equation}
\Sigma : u_3\ \left[u_6(\alpha u_1 + \beta u_2 -1) - u_3 (\alpha u_4 + \beta u_5) \right]=0.
\label{LP:sing}
\end{equation} 
\end{lemma}
\begin{proof}
From \cite{rasoulzadeh2019linear}, it is known that LP-case is generated whenever there is a \emph{singular affine map} $\kappa : (x_{i}, y_{i})\longrightarrow r_{i}$, where $r_{i}=\alpha x_{i}+\beta y_{i}$. Substituting this relation into $\det\left({S}\right)$ and factorizing, one obtains Eq.\,\ref{LP:sing}. 
\end{proof}
Later on, during the process of singularity-free path optimization, a projection of a point in $\mathbb{R}^6$ onto the variety is needed. Hence, a more detailed understanding of the $\Sigma$-variety is helpful. In the following theorem the properties of $\Sigma$-variety in the \emph{real space} are investigated.\\
\begin{theorem}
\label{theorem:properties}
The $\Sigma$-variety has the following properties: 
\begin{itemize}
\item[a)] $\Sigma$-variety is an algebraic variety formed by the union of a hyperplane $\Sigma_{1}$, and a hyperquadric $\Sigma_{2}$, in $\mathbb{R}^{6}$,
\item[b)] $\Sigma_2 = \Sigma_3 \cupdot M$, where $\Sigma_3$ is the set of singular points of $\Sigma_2$ and $M$ is a smooth manifold,
\item[c)] $\Sigma_{1}\cap\Sigma_{2}$ is a 4-dimensional algebraic variety, consisting of the union of two 4-planes $\mathcal{A}$ and $\mathcal{B}$,
\item[d)] $\Sigma_{1}$ is tangent to $\Sigma_{2}$ at a 3-dimensional smooth manifold $M^{\prime} \subset \Sigma_{1}\cap\Sigma_{2}$,
\item[e)] $\Sigma_3$ is a 2-plane and is contained in $\mathcal{A}\cap\mathcal{B} \subset \Sigma_{1}\cap\Sigma_{2}$,
\\
Finally, having in mind that $\mathbf{V}\left( f_1,\cdots,f_n\right)$ is the set of solutions to the system of polynomial equations $\langle f_1,\cdots,f_n\rangle$, the details are summarized below:
\begin{itemize}
\item{LO-case :}
    \subitem{$\Sigma_{1}=\mathbf{V}(u_6)$},
    \subitem{$\Sigma_{2}=\mathbf{V}\left( u_6(\alpha u_{1} + \beta u_{2}) - u_{3} (\alpha u_4 + \beta u_5 -1)\right)$},
    \subitem{$\Sigma_{3}=\mathbf{V}(\alpha u_{1} + \beta u_{2}, u_{3}, \alpha u_4 + \beta u_5-1, u_6)$}.
  \item{LP-case:}
    \subitem{$\Sigma_{1}=\mathbf{V}(u_{3})$},
    \subitem{$\Sigma_{2}=\mathbf{V}\left( u_6(\alpha u_{1} + \beta u_{2} -1) - u_{3} (\alpha u_4 + \beta u_5)\right)$},
    \subitem{$\Sigma_{3}=\mathbf{V}\left(\alpha u_{1} +\beta u_{2} -1, u_{3}, \alpha u_4+\beta u_5, u_6\right)$}.
\end{itemize}
\end{itemize}  
\end{theorem}
\begin{proof}
Although the proof is not complex, it might be perceived to be distinct from the rest of the paper. Consequently the interested reader can find the proof and related details in Appendix\,\ref{appendix:A}. 
\end{proof}
In Fig.\,\ref{fig:sigma}, an imaginative geometric depiction of $\Sigma$-variety in $\mathbb{R}^6$ is described in such a way that the there is a correspondence between the co-dimension of its properties and of the properties listed in Theorem\,\ref{theorem:properties}.
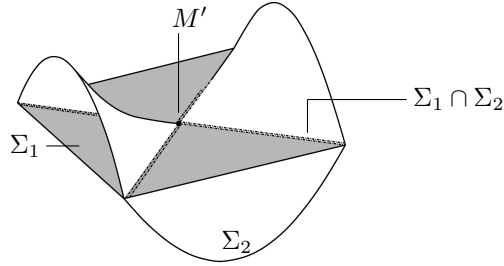
\begin{figure}[t!] 
\begin{center}   
   \begin{overpic}[height=35mm]{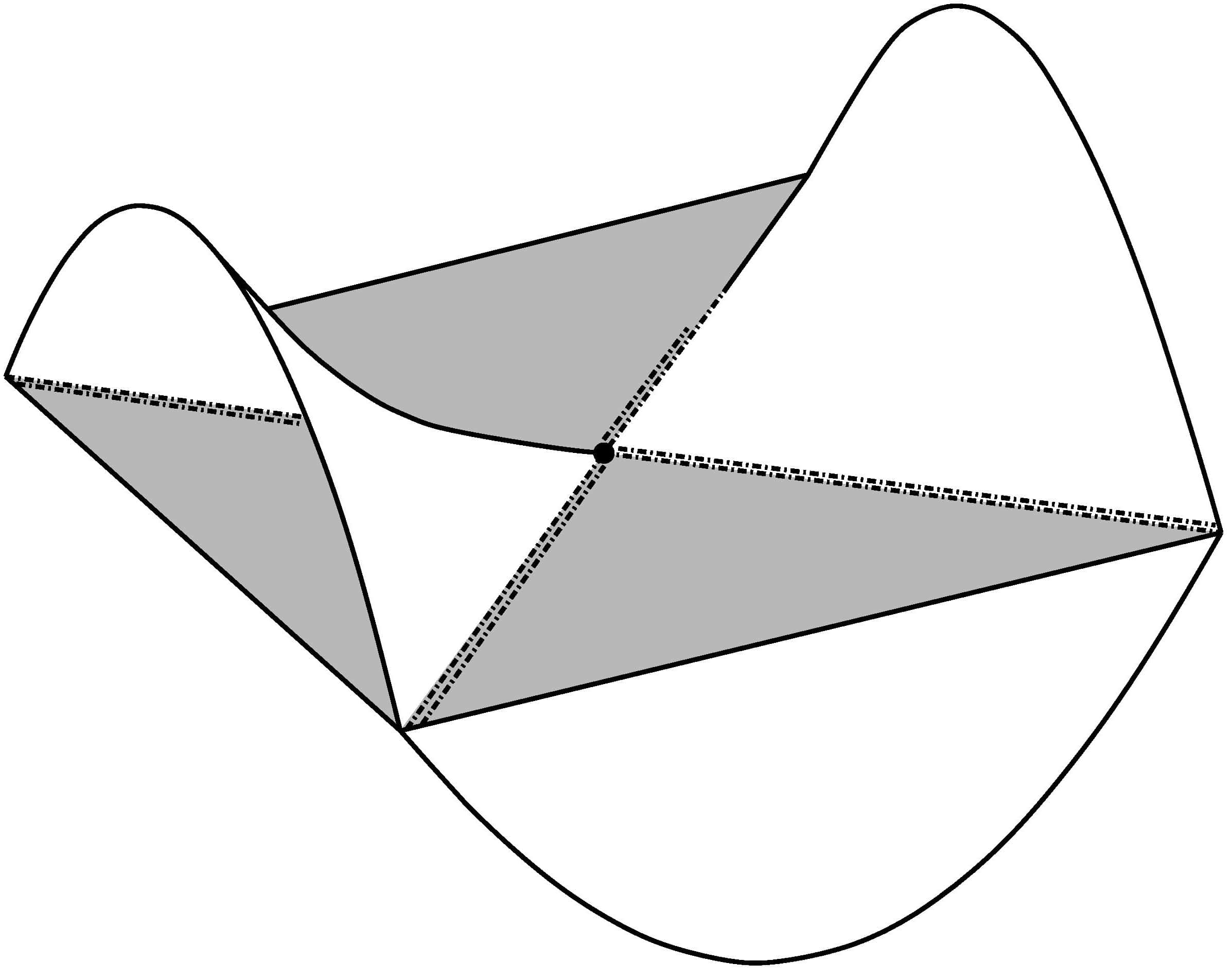}
	\begin{small}
   \put(88,50){\color{black}\line(1,0){30}}
   \put(88,50){\color{black}\line(0,-1){10}}
   \put(49,70){\color{black}\line(0,-1){24}}
   \put(47,72){$M^{\prime}$}
   \put(120,48){$\Sigma_1 \cap \Sigma_2$}
   \put(62,5){$\Sigma_2$}
   \put(-2,34){$\Sigma_1$}
   \put(9,36){\color{black}\line(1,0){10}}
   \end{small}     
   \end{overpic}
	\caption{One way to have an imaginative illustration of the $\Sigma$-variety is to consider a case with a certain similarity in $\mathbb{R}^3$. In order to do so, consider the variety $\mathbf{V}\left( z\,(x^2-y^2-z^2)\right)$ which is in fact the union of a plane $\Sigma_1$ and a hyperbolic paraboloid $\Sigma_2$. In that way, the $\Sigma_1 \cap \Sigma_2$ will be of co-dimension 2 while $\Sigma_1$ and $\Sigma_2$ are tangent at a 0 dimensional subvariety. Note that due to the fact that in this example we are restricted to $\mathbb{R}^3$ it is impossible to show the 2-dimensional singular set on $\Sigma_2$ as it is of co-dimension 4.}
	\label{fig:sigma}
\end{center}
\end{figure} 
%
%
\section{Differential Geometric Setup}
\label{sec:differentialgeometricsetup}
In this section, the \emph{nice motion} is revisited in a more rigorous format. Mathematically, the features of this motion are strongly connected to the concept of \emph{variations of energy} and consequently to the \emph{arc length} and \emph{curvature}. Since these topics are classified as differential geometry, a separate section is dedicated to them (as indeed we are viewing the problem from a different angle).\\
Evidently, to do differential geometric calculations an \emph{inner product} is required. Such an inner product should be in a way that one can derive kinematic information from. It turns out that it is possible to define a metric tensor (and consequently an inner product) in such a way that it implies the metric described in \cite{rasoulzadeh2018rational} (see Eq.\,\ref{distance}). Having access to the inner product it is possible to calculate the \emph{geodesic energy} (or simply just energy as it is more common in differential geometry texts) and \emph{bending energy}.  
\subsection{Metric Tensor}
Here, the goal is to frame-up the settings of differential geometry in order to be able to carry out the required computations on manifolds. The key element is to equip each tangent space (of the manifolds under study) with an \emph{inner product} that varies smoothly on the corresponding manifold. Such an inner product is obtained by computing the \emph{metric tensor} at one point of the manifold.  

\begin{definition}
The``object-oriented metric tensor" relative to the canonical basis of $\mathbb{R}^6$ is defined by the following matrix:    
\begin{equation}
g =\begin{small}\left( \begin {array}{cccccc} R& & &J& & \\ \noalign{\medskip} &R& & 
&J& \\ \noalign{\medskip} & &R& & &J\\ \noalign{\medskip}J& & &1& &
\\ \noalign{\medskip} &J& & &1& \\ \noalign{\medskip} & &J& & &1
\end {array} \right)_{6\times 6},\end{small}
\label{def:tensor}
\end{equation}
where $R=\frac{1}{5}(r_1^2+r_2^2 + r_3^2 +r_4^2 +r_5 ^2)$ and $J=\frac{1}{5}(r_1 + r_2 + r_3 + r_4 + r_5)$.
\label{def:metric:tensor}
\end{definition}
The following remark clarifies the relation between the object-oriented metric tensor and the object-oriented metric:
\begin{remark}
Consider the object-oriented metric $\mathfrak{d}$, Eq.\,\ref{distance}, and the object-oriented metric tensor $g$. For two points of $\mathbb{R}^6$ (two poses of the pentapod), namely, $u = \left( u_{1},u_{2},u_{3},u_{4},u_{5},u_{6}\right)$ and $p = \left( p_{1},p_{2},p_{3},p_{4},p_{5},p_{6}\right)$ one obtains (for more details cf.\,\cite{boothby1986introduction}): 
\begin{equation}
\mathfrak{d}(p,u)^{2}={C}^{T}.g.C,
\label{metric:metric-tensor}
\end{equation}
where ${C}^{T}=(u_1-p_1,u_2-p_2,u_3-p_3,u_4-p_4,u_5-p_5,u_6-p_6)^{T}$.
\label{remark:metric:metric-tensor}
\end{remark}
Note that the right hand side of Eq.\,\ref{metric:metric-tensor} serves as the desired \emph{inner product}:
\begin{gather}
\langle -, -\rangle : \mathbb{R}^{6}\times\mathbb{R}^{6}\longrightarrow \mathbb{R},\\
\left({X},{Y}\right)\longmapsto {X}^{T}.{g}.{Y},
\label{innerproduct}
\end{gather}
Then the pose space of the pentapod equipped with the inner product of Eq.\,\ref{innerproduct} forms the Riemannian manifold $\left( \mathbb{R}^6 , {g}\right)$.    
\subsection{Orthogonal Projection}
One of the fundamental tools in the variational path optimization approach is the \emph{orthogonal projection}.
\begin{definition}
%
Let $M$ be a non-empty subset of the metric space $\mathbb{R}^d$. The orthogonal projection into $M$, is the relation $\pi \subset \mathbb{R}^d \times M$ such that for each $p \in \mathbb{R}^d$ we have  
\begin{equation}
\pi\left( p\right) := \{ q \in M\ :\ \mathfrak{d}\left( p, q \right) = \mathcal{\varrho} \left( p , M \right) \}, 
\end{equation}
where $\mathfrak{d}$ is the metric of $\mathbb{R}^d$ and $\mathcal{\varrho} \left( p , M \right) := \inf_{q \in M} \mathfrak{d}(p , q)$.
\label{def:orthogonalprojection}
\end{definition}
The fact that $\pi$ is defined as a relation rather than a map is due to the fact that it might not be a function \cite{dudek1994nonlinear}.
\subsubsection{Computation of Pedal Points}
\label{section:pedal}
Definition\,\ref{def:orthogonalprojection} gives an intuition toward the nature of the orthogonal projection but does not provide a computational tool. In fact there can be several ways to calculate $\pi$ for a sample point $p\in\mathbb{R}^6$. But for reasons that will be cleared later the \emph{Lagrange multiplier method} is chosen. 
\begin{definition}
Consider the ``Lagrange equation" $L:= f + \lambda_1 \Phi_1 + ... + \lambda_n \Phi_n$, 
where  $f$ is the smooth ``distance function" on a manifold $N$. Moreover a certain subset $S$ of $N$ is given as the zero-set of the polynomials $\Phi_{1},\ldots,\Phi_{n}$ and their corresponding ``Lagrange multipliers" are denoted by $\lambda_1,\ldots,\lambda_n$. Then the solutions of the system of equations
\begin{equation}
  \nabla L = 0 \quad \text{and} \quad
  \Phi_i = 0 \quad \text{for} \quad i = 1, ... ,n
\label{definition:pedal}
\end{equation}
are called ``pedal points", as these points of $S$ to the given configuration ($N = \mathbb{R}^6$) of the linear pentapod cause local extrema of the distance function.
\label{def:pedal}
\end{definition}
Note that for $p\in\mathbb{R}^6$, the set $\pi\left( p \right)$ is a subset of \emph{pedal points}.\\
Previously the authors had announced that for an arbitrary point in $\mathbb{R}^6$, the number of corresponding pedal points on the $\Sigma$-variety is 3 (cf. \cite{rasoulzadeh2019linear2}). In this section, the goal is to obtain closed-form information on the procedure to compute these pedal points.
\begin{theorem}
Consider an arbitrary point $p$ in $\mathbb{R}^6$. Then with respect to the object-oriented metric, Eq.\,\ref{metric:metric-tensor}, the corresponding number of pedal points on the $\Sigma$-variety is up to 3 and the pedal point coordinates can be obtained in a closed form. 
\label{theorem:pedals}
\end{theorem}
\begin{proof}
In Theorem \ref{theorem:properties}, it is shown that the $\Sigma$-variety decomposes into the hyperplane $\Sigma_1$ and the hyperquadric $\Sigma_2$. Trivially, one of the pedal points is located on $\Sigma_1$, namely, the \emph{closest} point (with respect to the distance function) to the arbitrary point $p$. Now, knowing this fact it is possible to focus on $\Sigma_{2}$ for retrieving the pedal points information.
Consider the following \emph{Lagrange equation}
\begin{equation}
L:= \mathfrak{d}^2\left( p, u\right) + \lambda . f
\label{eq:hyperquadricpedal}
\end{equation} 
where $\lambda$ is the \emph{Lagrange multiplier}, $f$ is the singularity polynomial of $\Sigma_2$ and $\mathfrak{d}$ is the \emph{object-oriented metric} describing the distance between the point ${p}$ and a sample point ${u}$ on $\Sigma_2$ with symbolic coordinates $(u_1,u_2,u_3,u_4,u_5,u_6)$. Then, computing 
\begin{equation}
I :=
\langle
 \frac{\partial L}{\partial u_1}, \frac{\partial L}{\partial u_2},\ldots , \frac{\partial L}{\partial u_6}, \frac{\partial L}{\partial \lambda}
 \rangle
\label{ideal}
\end{equation}
results in a system of polynomial equations, where ${\partial L}/{\partial u_1},\ldots,{\partial L}/{\partial u_6}$ are linear in the pose variables. These pose variables depend on $\lambda$. By solving them for $\lambda$ and substituting them in ${\partial L}/{\partial \lambda}$ one obtains a univariable polynomial quadratic in $\lambda$. Since this polynomial is of degree $2$ it shows that for the arbitrary point ${p}$ there will be at most $2$ pedal points on $\Sigma_2$. Finally, the pedal point coordinates are obtained through solving the quadratic polynomial for $\lambda$ and back-substituting its value.
\end{proof}
\subsubsection{Distance to $\Sigma$-variety \& Modified Orthogonal Projection}
\label{gradient:lines}
Obtaining the closed form pedal point coordinates, enables one to compute the vector field of normals to $\Sigma$-variety at these points. Assuming a pose point $p\in\mathbb{R}^6$, the \emph{natural} way of increasing distance to the $\Sigma$-variety is to move on the line spanned by $p$ and the closest pedal (from now on called \emph{gradient lines}). However, this natural primary idea needs to be modified as the following cases may occur:
\vspace*{2mm}
\begin{itemize}
\item[$\bullet$] \textbf{Pedal points on} $\Sigma_3$:\\
Despite the fact that, based on Definition\,\ref{def:orthogonalprojection}, there is no restriction for the closest point to any specific subset of the $\Sigma$-variety, if such a point is located on $\Sigma_3$, Definition\,\ref{def:pedal} does not reveal it. The reason behind it is that the Lagrange multiplier technique nourishes from the tangency of the \emph{distance function} and the \emph{constraint} at the pedals. In another word, such minimizer is found whenever the corresponding gradients of the distance function and the constraint are linearly dependent. On the other hand, at singular points of the constraint the gradient vanishes which results the incapability of the Lagrange multipliers in identifying pedals on $\Sigma_3$. One might try to overcome this hardship through redefining the pedal points as the solutions to the system of polynomials obtained by calculating the derivatives of the \emph{distance function} with respect to parameters of $\Sigma_2$. In such a method one obtains the pedals on $\Sigma_3$ at the cost of expensive symbolic computation of non-linear systems. Though this technique is not reasonable, resolving this issue is not hard as one can find the closest point on $\Sigma_3$ separately. Hence, from now on we loosely call this point a pedal point as well (for details of $\Sigma_3$ cf. Appendix\,\ref{appendix:A}).\\
\item[$\bullet$] \textbf{Pose points on cut locus}:\\
By Theorem\,\ref{theorem:pedals} and \emph{previous case}, we know that $p$ can be located on up to 3 gradient lines plus a line connecting $p$ to the pedal point on $\Sigma_3$. A natural question would be which line to go along with to increase the distance to $\Sigma$. One idea would be to consider taxi along the the lines connecting pose point to the \emph{closest} pedal. But if the pose point is located on the \emph{cut locus} of $\Sigma$ then more than one closest point already exist (for definition and properties of cut locus see \cite{carmo1992riemannian}).
On the other hand, as pose point updating is discrete, with respect to a step-size, it may create problems such as overshooting over the cut locus. Such overshooting happens due to a ``jump" of the corresponding closest pedal. In order to resolve both issues at once and create a smooth update of the pose point, a \emph{weight} factor is introduced in such a way that it considers the effect of the pedal points in dependence of their distances to the pose $p$.  
\end{itemize}
\vspace*{2mm}
Due to the reasons mentioned above, we modify Definition\,\ref{def:orthogonalprojection} into the following shape:
\begin{definition}
The orthogonal projection with respect to object-oriented metric is a relation $\pi\subset\Sigma\times\mathbb{R}^6$, such that for each $p\in\mathbb{R}^6$ we have
\begin{equation}
\pi\left( p\right) := \{ \textit{closest point on}\ \Sigma_1 \} \cup \mathbf{V}\left( I \right) \cup \{ \textit{closest point on}\ \Sigma_3 \}, 
\label{eq:modifiedorthogonal}
\end{equation}
where $I$ is Eq.\,\ref{ideal}.
\label{def:modifiedorthogonal}
\end{definition}
Note that in a generic case, $\pi\left( p\right)$ is finite and yields four pedal points (one on $\Sigma_1$, two on $\Sigma_2$ and one on $\Sigma_3$).
\begin{remark}
Note that, it is not immediately known whether the non-real complex pedal points or double pedal points exist or not. Calculating Eq.\,\ref{ideal} for the case of a \emph{parabola} as an obstacle shows that the pedal points are not always distinct real points. 
\label{remark:pedals}
\end{remark}
\begin{definition}
Assume $p \in \mathbb{R}^6$ and $\pi$ to be an orthogonal projection relation mentioned above. Additionally, assume $m_p$ to be the number of pedal points on $\Sigma$ with respect to the point $p$. Now, the goal is to find a preferred direction such that if point $p$ travels along with, it can increase its distance to the $\Sigma$-variety. Assuming $u$ to be the update of $p$, the following function fulfils the goal:   
\begin{equation}
\mathcal{D}\left(p, u\right) := -\sum_{i=1}^{m_p} w_i \langle\,\frac{p - {\pi_{i}\left( p \right)}}{\| p - {\pi_{i}\left( p \right)}\|} , u-p\,\rangle.
\label{cost:1}
\end{equation}
where the $w_i := H / d\left(p, q_i\right)$ are the ``positive weights" and $H$ given by
\begin{equation}
H:= \frac{\prod_{i=1}^{m_p} \mathfrak{d}\left(p, q_{i}\right)}{\sum_{j=1}^{m_p}\prod_{k\neq j} \mathfrak{d}\left(p, q_{k}\right)}
\label{eq:H}
\end{equation}  
\label{def:weight} 
\end{definition}
is the ``regulation factor" implying $\sum_{i=1}^{m_p}w_i = 1$.\\
Eq.\,\ref{cost:1} refers to a weighted projection of vector $\overrightarrow{pu}$ on the three gradient lines plus the line connecting $p$ to the closest pedal on $\Sigma_3$. As it can be inferred from Eq.\,\ref{eq:H}, these weights are disproportional to the corresponding distance to pedals and their sum is equal to one. In another word, if $p$ is far away from a pedal then the projection of movement direction, $\overrightarrow{pu}$, on that gradient is a smaller vector (going away from that pedal is less important). On the contrary, if $p$ is close to a pedal then the projection of $\overrightarrow{pu}$ would be mainly along that gradient (it will repel the point $p$ with more efficiency). In Fig.\,\ref{fig:threepics}-left an example regarding such weighted factors is depicted.\\
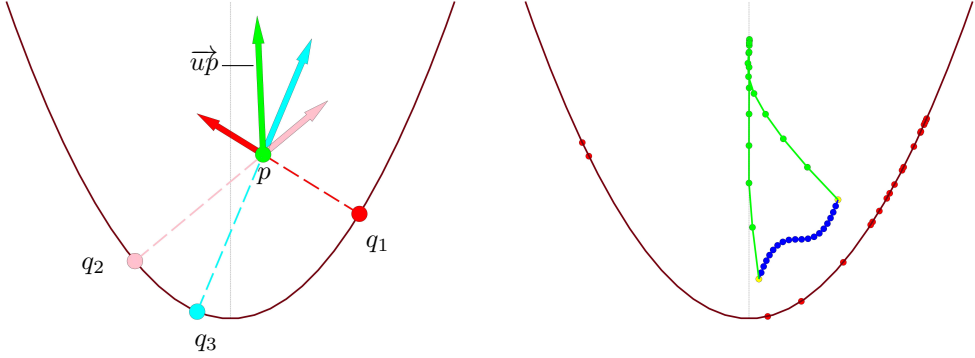
\begin{figure}[t!] 
\begin{center}
   \begin{overpic}[height=60mm]{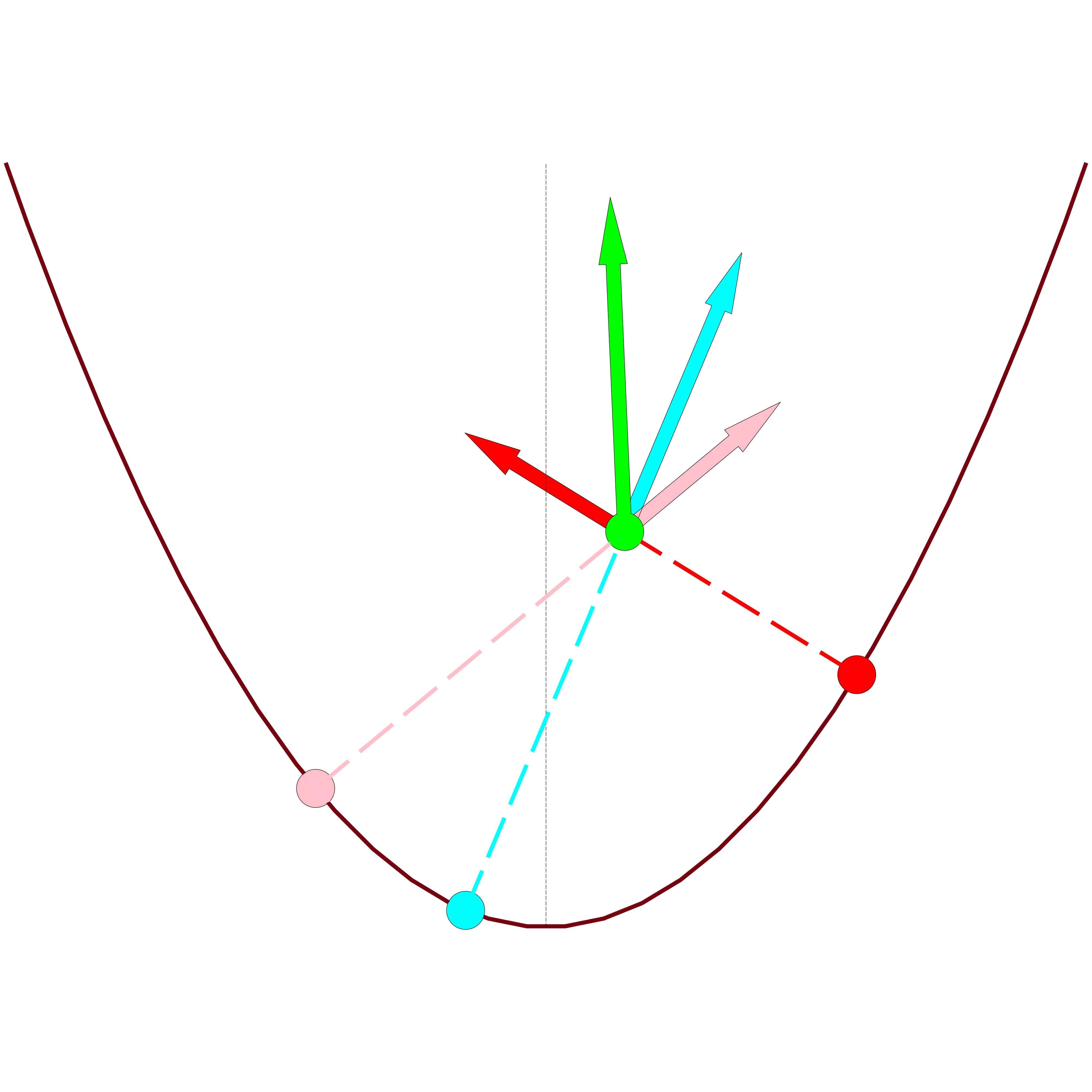}
   \begin{small}
   \put(55.1,71){\color{black}\line(-1,0){7}}
   \put(56,46){$p$}
   \put(80,31){$q_1$}
   \put(42,9){$q_3$}
   \put(17,26){$q_2$}
   \put(41,70){$\overrightarrow{up}$}
   \end{small}     
   \end{overpic}
   \hspace*{7mm}
   \includegraphics[trim=0mm 0mm 0mm 0mm,clip,height=60mm]{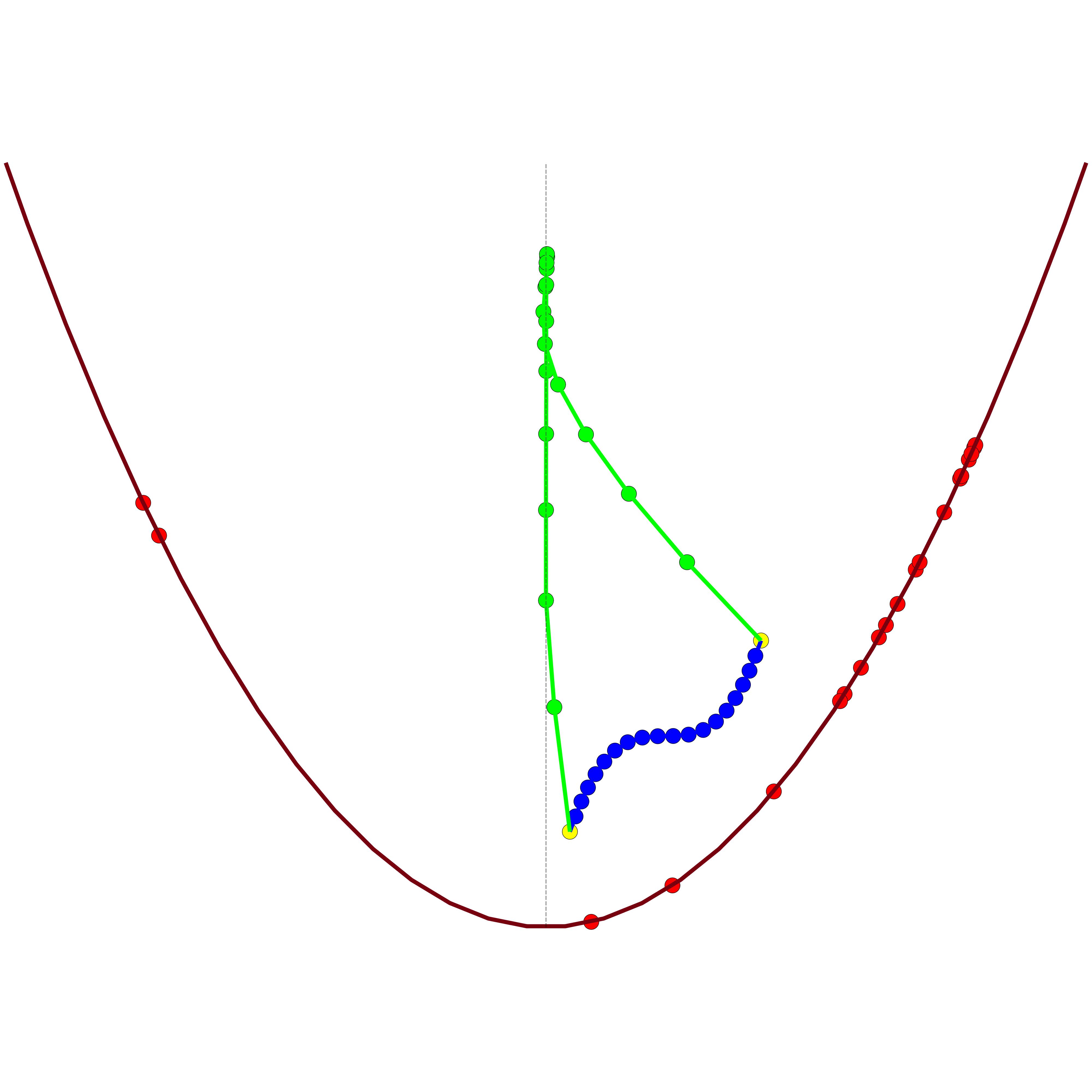}
   \caption{Left: The green vector shows the optimized direction for the vector $u-p$. The projection on the 1st, 2nd and 3rd gradient lines are indicated by red, pink and cyan respectively. Note that, one can check the correctness of the optimized move as the green direction refers to moving closer to the medial axis and increasing distance to the parabola eventually. 
   Right: Depending on the geometrical structure of the obstacle, geodesic and bending weight factors, the algorithm may not have a final result. The initial path (blue) variates (green) along the parabola's medial axis till infinity under zero geodesic and bending weight inputs.}
\label{fig:threepics}
\end{center}
\end{figure} 
\subsection{Variations of Energy}
\begin{definition}
Let $c:\left[ a, b\right]\longrightarrow N$ be piecewise smooth curve on a Riemannian manifold $N$. A ``variation" of $c$ is a continuous mapping $f:\left(-\epsilon ,\epsilon\right) \times \left[a,b\right] \longrightarrow N$ such that\\ 
\textit{a)} $f(0,t) = c(t)$, $t \in \left[ a, b\right]$,\\
\textit{b)} there exists a subdivision of $\left[ a,b\right]$ by points $a= t_1 < ... < t_{n-1} < t_n = b$ in a way that the restriction of $f$ to each $\left(-\epsilon ,\epsilon\right) \times \left[t_i,t_{i+1}\right]$ is differentiable \cite{carmo1992riemannian}. 
\label{def:variation}
\end{definition}
A variation is said to be \emph{proper} if $f\left(s,a\right) = c\left(a\right)$ and $f\left(s,b\right) = c\left(b\right)$ for all $s\in\left(-\epsilon, \epsilon\right)$ \cite{carmo1992riemannian}.
Considering a smooth variation (smooth $f$), the \emph{geodesic energy} $E$ and \emph{bending energy} $B$ of the variation $f$ for $s\in\left(-\epsilon,\epsilon\right)$ are defined as follows:
\begin{equation}
E \left( s\right) = \int_{a}^{b} \| \frac{\partial\,f}{\partial\,t} \left( s,t\right)\| ^ 2\ dt,\,\,\,\,\,\,\,\,\,\, B \left( s \right) = \int^{b}_{a} \parallel \frac{\partial^2\,f}{\partial\,t^2} \left( s,t\right) \parallel ^ 2 \ dt.
\label{eq:continuousenergies}
\end{equation} 
Later it will be shown that an \emph{update} of the \emph{initial curve} is discretely done and depends on a predefined \emph{step size} (also known as \emph{learning rate}). Hence, a modification of Eq.\,\ref{eq:continuousenergies} into our discrete setting is necessary.
\subsubsection{Geodesic Energy \& Bending Energy}
\label{sec:geodesicbending}
For the sake of simplicity of notations, we restrict ourselves to the case of the piecewise smooth curve in $\mathbb{R}^6$.
\begin{definition}  
Assume that $c :\left[ a, b \right] \longrightarrow \mathbb{R}^6$ is a ``piecewise smooth" curve with a finite partition, $a= t_1 < ... < t_{n-1} < t_n = b$, then the ``discrete geodesic energy" $E$ and ``discrete bending energy" $B$ of the curve $c$ are as follows (cf. \cite{pottmann2009computational}):
\begin{equation}
E \left( c \right) = \sum_{i = 2}^{n} \| c \left( t_i \right) - c \left( t_{i-1} \right)\| ^ 2,\,\,\,\,\,\,\,\,\,\, B \left( c \right) = \sum_{i = 2}^{n-1} \|  \left( c \left(t_{i+1}\right) - c\left(t_{i}\right) \right) - \left( c \left(t_i\right) - c\left(t_{i-1}\right) \right) \| ^ 2.
\label{energy:discrete}
\end{equation}
while ``discrete length" and ``discrete total curvature" are respectively given by (cf.\,\cite{carmo1992riemannian, milnor1950total}):
\begin{equation}
L \left( c \right) = \sum_{i = 2}^{n} \| c \left( t_i \right) - c \left( t_{i-1} \right)\| ,\,\,\,\,\,\,\,\,\,\, \tau \left( c \right) = \sum_{i = 2}^{n-1} \|  \left( c \left(t_{i+1}\right) - c\left(t_{i}\right) \right) - \left( c \left(t_i\right) - c\left(t_{i-1}\right) \right) \| .
\end{equation}
\label{def:geodesic:bending}
\end{definition}
Finally, it is noteworthy to emphasize that the variations under study in this paper are \emph{proper} as we require the motions to remain between a fixed start- and end-pose.
\subsection{Cost Function}
\label{sec:costfunction}
In order to find the \emph{nice motion} the idea is to use \emph{gradient descent} which is a first-order iterative optimization algorithm to find the local minimum of a \emph{cost function}. Though in many similar applications, especially in the context of machine learning, the step size (learning rate) is considered a constant small number, here it will not be accurate enough and hence a more intelligent step should be taken for that purpose. Additionally, the function whose descent per iterations is monitored is slightly different with the original cost function and will be introduced under the name of \emph{objective function}.\\
\begin{definition}
Assume that an initial singularity-free curve between start-pose $a^{\prime}$ and end-pose $b^{\prime}$ is given, while it is discretized into a ``piecewise smooth" curve with the set of $n$ breakpoints $\{ a^{\prime} = p^{1}, p^{2}, ..., p^{n} = b^{\prime}\}$ in $\mathbb{R}^6$. Assuming $u^j$ to be the update of $p^j$, the cost function, subject to optimization, is defined as follows:
\begin{equation}
\mathcal{C} \left( p, u \right) :=
\frac{\lambda\,(n-1)}{2\,L^{\prime}}\,E\left( u\right) + \frac{\eta\,(n-2)}{2\,\tau^{\prime}}\,B\left( u\right) - \frac{1}{(n-2)}\sum_{j=2}^{n-1} \mathcal{D}\left(p^j, u^j\right),
\label{eq:costfunction}
\end{equation}
where $\mathcal{D}$ is the function of Eq.\,\ref{cost:1}, $\lambda$ and $\eta$ are real numbers called ``geodesic weight" and ``bending weight", respectively and $L^{\prime}$ and $\tau^{\prime}$ are the discrete length and discrete total curvature obtained from previous iteration (at the first iteration they are substituted by the corresponding values of the initial curve). Finally, we have $u^T := \left( (u^2)^T,...,(u^{n-1})^T\right)$ and $p^T := \left( (p^2)^T,...,(p^{n-1})^T\right)$.
\label{def:costfunction}
\end{definition}
Geodesic and bending weights increase/decrease the geodesic energy and bending energy by blocking/unblocking their increase when substituted by high/low values.\\
\begin{remark}
In the absence of the coefficients $1/L^{\prime}$ and $1/\tau^{\prime}$ in Eq.\,\ref{eq:costfunction}, the geodesic energy and bending energy terms which are the sum of the \emph{squared values} would heavily outgrow the distance term containing the projection on the gradient lines which is single valued (cf. Fig.\,\ref{fig:areal}). This fact forces the cost function to be heavily dependant on the number of chosen breakpoints. Consequently, the existence of these coefficients are necessary.
\end{remark}
\begin{remark} 
It is noteworthy that the coefficients of the geodesic and bending energy terms in Eq.\,\ref{eq:costfunction} are multiplied by a factor involving the number of break points while the projection term's coefficient plays the role of a mean value for the number of breakpoints. These coefficients are designed in this way to reduce the effect of the number of breakpoints on the shape of the optimized curve. This is important as later it is shown that the number of break points may vary at different iterations of the optimization algorithm.
\end{remark}
The cost function, $\mathcal{C}$, is a \emph{quadratic} polynomial in $6\,(n-2)$ variables. Optimizing Eq.\,\ref{eq:costfunction} requires solving the following equation
\begin{equation}
\nabla\,\mathcal{C} = 0,
\label{gradienteq}
\end{equation}
for $u\in\mathbb{R}^{6\,(n-2)}$ which is equivalent to solving a linear system.  
\subsection{Step Size}
\label{sec:stepsize}
The idea is to calculate two step sizes, namely $s_1$ and $s_2$, in such a way that they yield $\pm\,\mathsf{growth}\,\%$ of variations of geodesic and bending energy respectively. If we define the geodesic energy of the updated breakpoints by:
\begin{equation}
P(s_1) := \sum_{i = 2}^{n} \| \left( p^i + s_1\,u^i \right) - \left( p^{i-1} + s_1\,u^{i-1} \right) \| ^ 2, 
\end{equation}
and its bending energy by
\begin{equation}
Q(s_2) := \sum_{j = 2}^{n-1} \| \left( p^{j-1} + s_2\,u^{j-1} \right) + \left( p^{j+1} + s_2\,u^{j+1} \right) - 2\,\left( p^{j} + s_2\,u^{j} \right) \| ^ 2,
\end{equation}
and having the previous geodesic energy $E^{\prime}$ and previous bending energy $B^{\prime}$, then the real roots of the following equations will guarantee $\pm \mathsf{growth}\%$ variation in energies:
\begin{equation}
P(s_1) - (1 \pm\,\frac{\mathsf{growth}}{100})\,E^{\prime} = 0,\ \ \ Q(s_2) - (1 \pm\,\frac{\mathsf{growth}}{100})\,B^{\prime} = 0.
\label{eq:variation}
\end{equation}
Finally, the step size is chosen as the minimum of the positive real roots of Eq.\,\ref{eq:variation} and $1$.
\begin{figure}[t!] 
\begin{center}   
   \begin{overpic}[height=50mm]{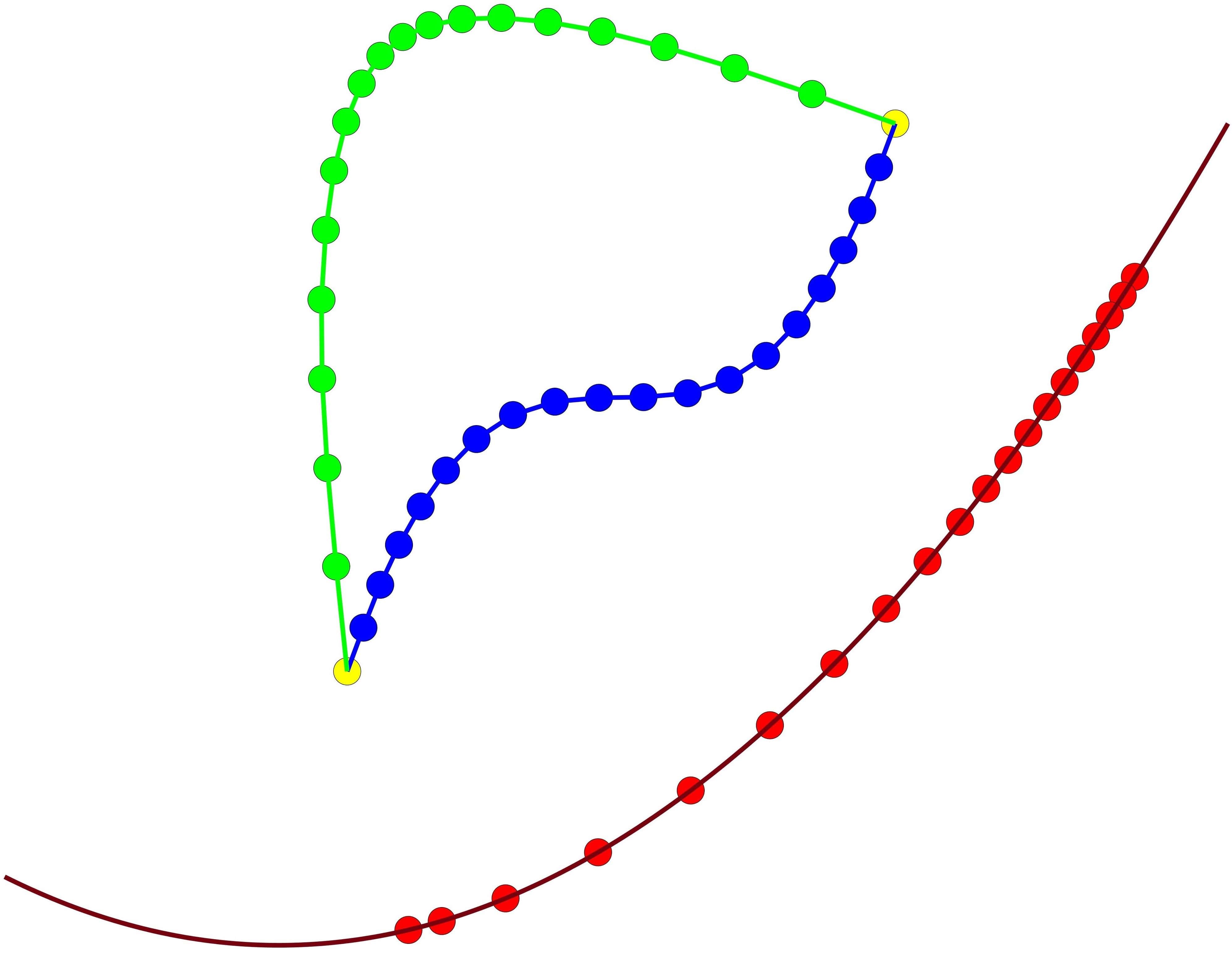}
	\begin{small}
   \put(102,67){$\Sigma$}
   \put(20,22){$a^{\prime}$}
   \put(75,70){$b^{\prime}$}
   \end{small}     
   \end{overpic}
   \hspace*{10mm}
   \begin{overpic}[height=50mm]{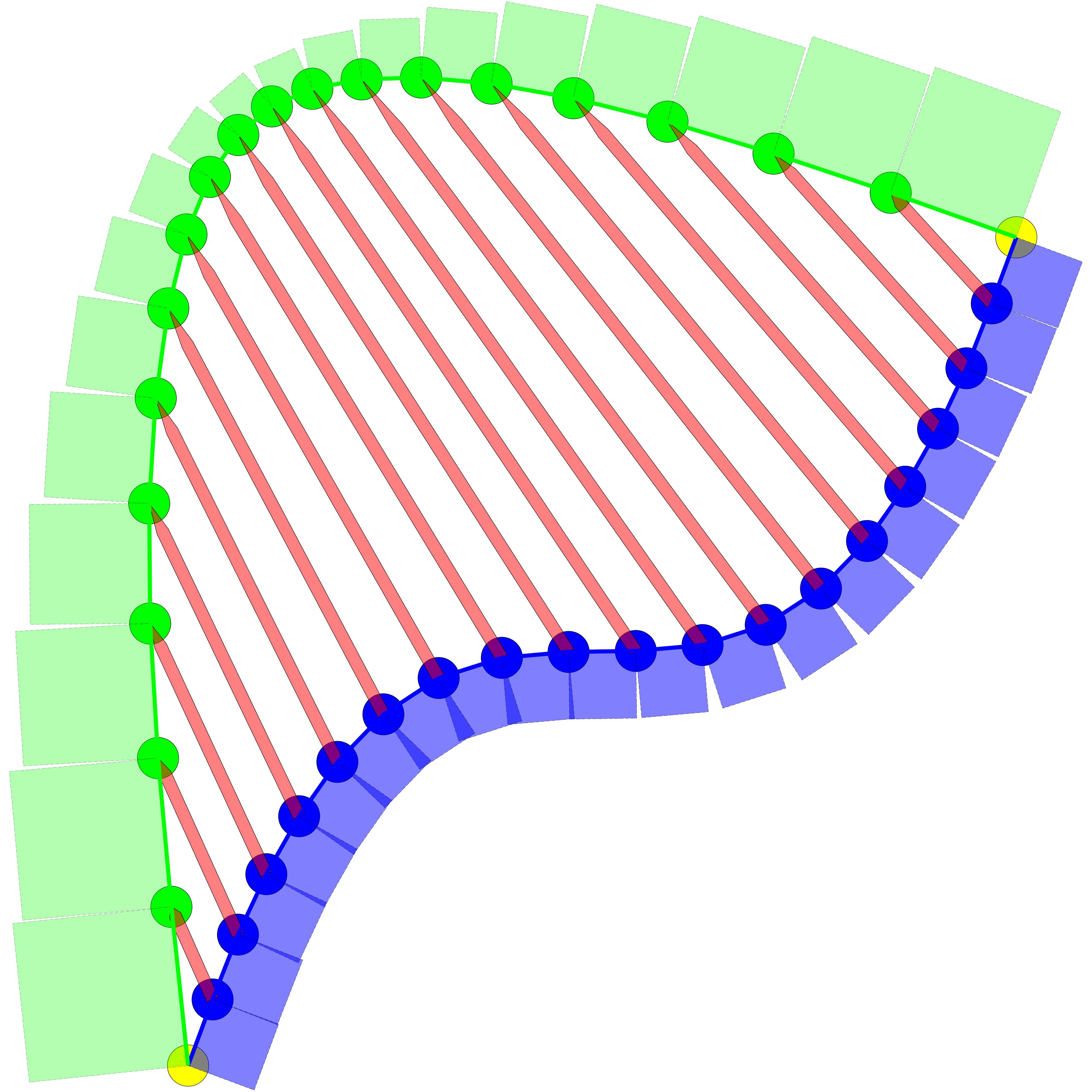}
	\begin{small}
	\end{small}     
   \end{overpic}
	\caption{\scriptsize Left: Illustration of the pull effect of geodesic energy on the breakpoints. The pattern of the breakpoints preserves a regular shape after variation. 
Right: A geometrical visualization of the errors. The blue boxes describe the calculated value for the geodesic energy before variation, while the green ones depict this energy after the variation. The red arrows connect a breakpoint to its corresponding update. One way to realize the error's occurrence would be through assigning physical units to the calculations, namely $cm$. In that way the the value computed for the geodesic energy will be of physical dimension $cm^2$ while the distance is $cm$.
Both figures are created for the case of the variation of an initial curve $t\mapsto \left(\,t + 1.25, t^3 + 2\right)$ where $t\in\left[\,-1,1\right]$ with respect to a parabola and with $\lambda =\eta = 10$ as weight factors.}
\label{fig:areal}
\end{center}
\end{figure} 
\subsection{Objective Function}
Now, in order to monitor the descent of the cost function with respect to iteration we propose a slightly different function called \emph{objective function}.
\begin{definition}
Assume $u$ to be the vector obtained by solving the linear system (cf. Section\,\ref{sec:costfunction}) and $\pi$ the orthogonal projection into the $\Sigma$-variety. Then the ``objective function" for monitoring the gradient descent is
\begin{equation} 
\mathcal{C}^{\prime}\left( u, p\right) := \frac{\lambda\,\left(n-1\right)}{2\,L}\,E\left({u}\right) + \frac{\eta\,\left(n-2\right)}{2\,\tau} \,B\left({u}\right) - \frac{1}{\left( n-2\right)}\,\sum_{j=2}^{n-1}\,\mathfrak{d}\left(\,u^j,\,\pi\left(\,p^{j}\,\right)\right).
\label{alt:1}
\end{equation}
\label{def:alt:1}
\end{definition}
Note that, as we have the updated curve breakpoint coordinates it is possible to compute length and total curvature now. Hence, the terms $1/L$ and $1/\tau$ replace the terms $1/L^{\prime}$ and $1/\tau^{\prime}$ from Eq.\,\ref{eq:costfunction} respectively. Additionally, since the goal is to increase the distance to the $\Sigma$-variety, the corresponding terms are replaced by the sum of distances to corresponding closest pedals.
\subsection{Remarks on the Curve's Energies}
\label{section:deviation}
Let $c: \left(a,b\right)\longrightarrow\mathbb{R}^n$ be a curve parametrized by \emph{arc length} $s$. Since the tangent vector $c^{\prime}\left( s\right)$ has unit length, the norm $\|c^{\prime\prime}\|$ of the second derivative measures the rate of change of the angle which neighboring tangents make with the tangent at $s$. $\|c^{\prime\prime}\|$ gives, therefore, a measure of how rapidly the curve pulls away from the tangent line at $s$, in the neighborhood of $s$ (namely, the curvature $\kappa\left( s\right)$)\cite{do2016differential}. Consequently, in the discrete case, if $\left\lbrace a^{\prime}=p_{1}, p_{2}, ..., p_{n}=b^{\prime}\right\rbrace$ is a sequence of breakpoints defining a piecewise smooth curve, it is natural to consider the angle $\theta_{i}$ between vectors $p_{i+1}-p_{i}$ and $p_{i}-p_{i-1}$ as the \emph{curvature} at the breakpoint $p_{i}$ and the sum $\sum^{n}_{i=1}\theta_{i}$ as the \emph{total curvature} \cite{milnor1950total}. In fact in this way one can define the bending energy by $\sum^{n}_{i=1}\theta^2_{i}$, but Definition\,\ref{def:geodesic:bending} is slightly different. Due to the fact that the measure of a central angle and the arc it intercepts are equal in value, in the case that the curve is parametrized by arc length the two definitions are equivalent. However, such values obtained may differ if the curve under consideration is not arc length parametrized (i.e. consider a line not possessing arc length parametrization, then based on Definition\,\ref{def:geodesic:bending} the bending energy is not zero). But such a deviation creates a computational advantage as otherwise a denominator containing the norms of every two adjacent edges would appear in the bending energy term $B$ of the cost function (Eq.\,\ref{def:costfunction}). This denominator then diminishes bending energy term being a quadratic polynomial and consequently denying the advantage of possessing a linear system after derivation. Finally, the geodesic energy term contributes significantly to resolving this issue as it creates a \emph{pull} effect which preserves a regular distribution of breakpoints across the curve after variation (see Fig.\,\ref{fig:areal}-right). 
\subsection{Orthogonal projection into the configuration space}
\label{sec:projection:cylinder}
In the process of updating breakpoints, by obtaining the preferred direction and step size (cf. Eq.\,\ref{eq:variation}), one may end up with a curve not necessarily belonging to the cylinder $\Gamma$. Having in mind that $\Gamma$ is the set of all performable motions for the pentapod, it is important to find an update restricted to it. In order to resolve it, one immediate idea is to keep the update of each breakpoint $p$ restricted to its corresponding tangent space $T_p\left(\Gamma\right)$ to the cylinder $\Gamma$. In the next step, the updated curve is orthogonally projected into the cylinder.
\subsubsection{Projection on Cylinder's Tangents}
For a breakpoint $p$ we are interested in its update $u$ belonging to $T_{p}(\Gamma)$. The following lemma creates the first step in building a systematic update of $\left\lbrace p^2, \ldots,p^{n-1}\right\rbrace$ on $T\,\Gamma$.   
\begin{lemma}
Let $f$ be a smooth function on a Riemannian manifold $N$. Let $M$ be a submanifold of $N$. Then the gradient of the map $f$ restricted to $M$ at a point $p\in M$, $\left(\nabla f|_M\right)_p$, is the orthogonal projection of $\left(\nabla f\right)_p$ onto $T_p\left( M\right)$.	
\label{lemma:oneil}
\end{lemma}
\begin{proof}
If $g$ is the \emph{Riemannian metric} on $N$ then $g|_{M}\left(\nabla f|_M, X\right) = X\left(f|_M\right)$ for all $X \in \mathfrak{X}\left( M\right)$ denoting the set of all smooth vector fields on $M$ (cf. \cite{o1983semi}). Now, considering the decomposition at the point $p\in M$, $\left(\nabla\,f\right)_p = v_{o} + v_{t}$ where $v_t \in T_{p}\left(M\right)$ and $v_o \in T_{p}\left(M\right)^{\perp}$ we have \footnote{In the following equation the subscript $p$ is dropped intentionally.}
\begin{eqnarray*}
g|_{M}\left(\nabla f|_M, X\right) = X\left(f|_M\right) = X\left(f\right) = g\left(\nabla\,f, X\right) = g\left( v_o + v_t, X \right) =\\
g\left( v_o, X\right) + g\left( v_t, X\right) = g\left( v_t, X\right).
\end{eqnarray*} 
which gives
\begin{equation}
g\left( \nabla\,f|_M, -\right) = g\left(v_t, -\right).
\end{equation}
\end{proof}
Each breakpoint of the piecewise smooth curve is located in the ambient space $\mathbb{R}^6$. Define 
\begin{equation}
{\hat{p}} = \left(  \begin{array}{cccc}\,
\left({p}^{\ 2}_{\ 6\times1}\,\right)^{T},\, 
\left({p}^{\ 3}_{\ 6\times1}\,\right)^{T},\,
\cdots ,\,
\left({p}^{\ n-2}_{\ 6\times1}\,\right)^{T}\, 
\end{array}\right)^{T}_{6 (n-2) \times 1},
\label{p:hat}
\end{equation}
where ${p}^{\ i}_{\ 6\times1}$ denotes the 6-dimensional coordinate vector of the $i$-th breakpoint of the piecewise smooth curve. In this way the point $\hat{p}\in\mathbb{R}^{\,6(n-2)}$ stands for the piecewise smooth curve. Defining the updated breakpoint $u$ in exactly the same fashion we have
\begin{equation}
{\hat{u}} = \left(  \begin{array}{cccc}\,
\left({u}^{\ 2}_{\ 6\times1}\,\right)^{T},\, 
\left({u}^{\ 3}_{\ 6\times1}\,\right)^{T},\,
\cdots ,\,
\left({u}^{\ n-1}_{\ 6\times1}\,\right)^{T}\, 
\end{array}\right)^{T}_{6 (n-2) \times 1}.
\label{u:hat}
\end{equation}
Now, using the above terminology one trivially finds that the vector $\hat{u} - \hat{p}$ is in fact the $\left(\nabla\mathcal{C}\right)_{\hat{p}}$ on the level sets of the cost function in $\mathbb{R}^{\,6(n-2)}$. 
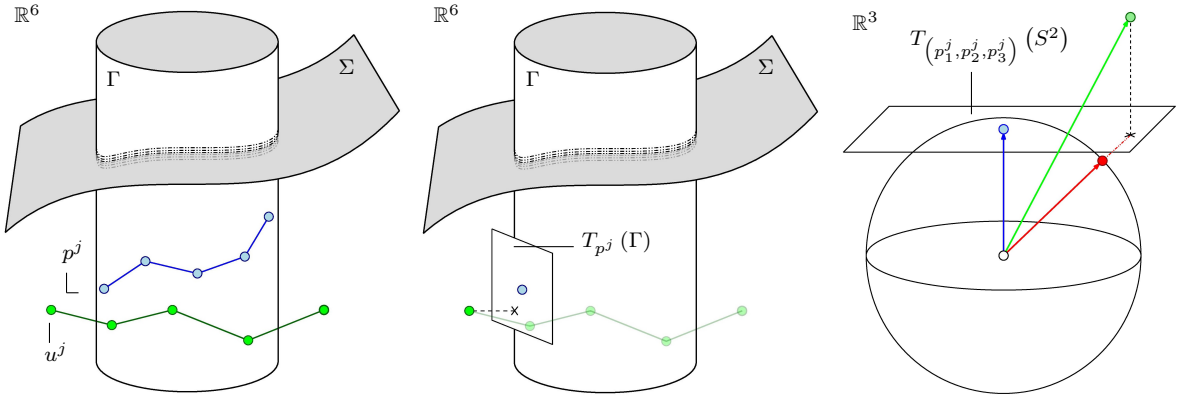
\begin{figure}[t!] 
\begin{center}   
  \begin{overpic}[height=51 mm]{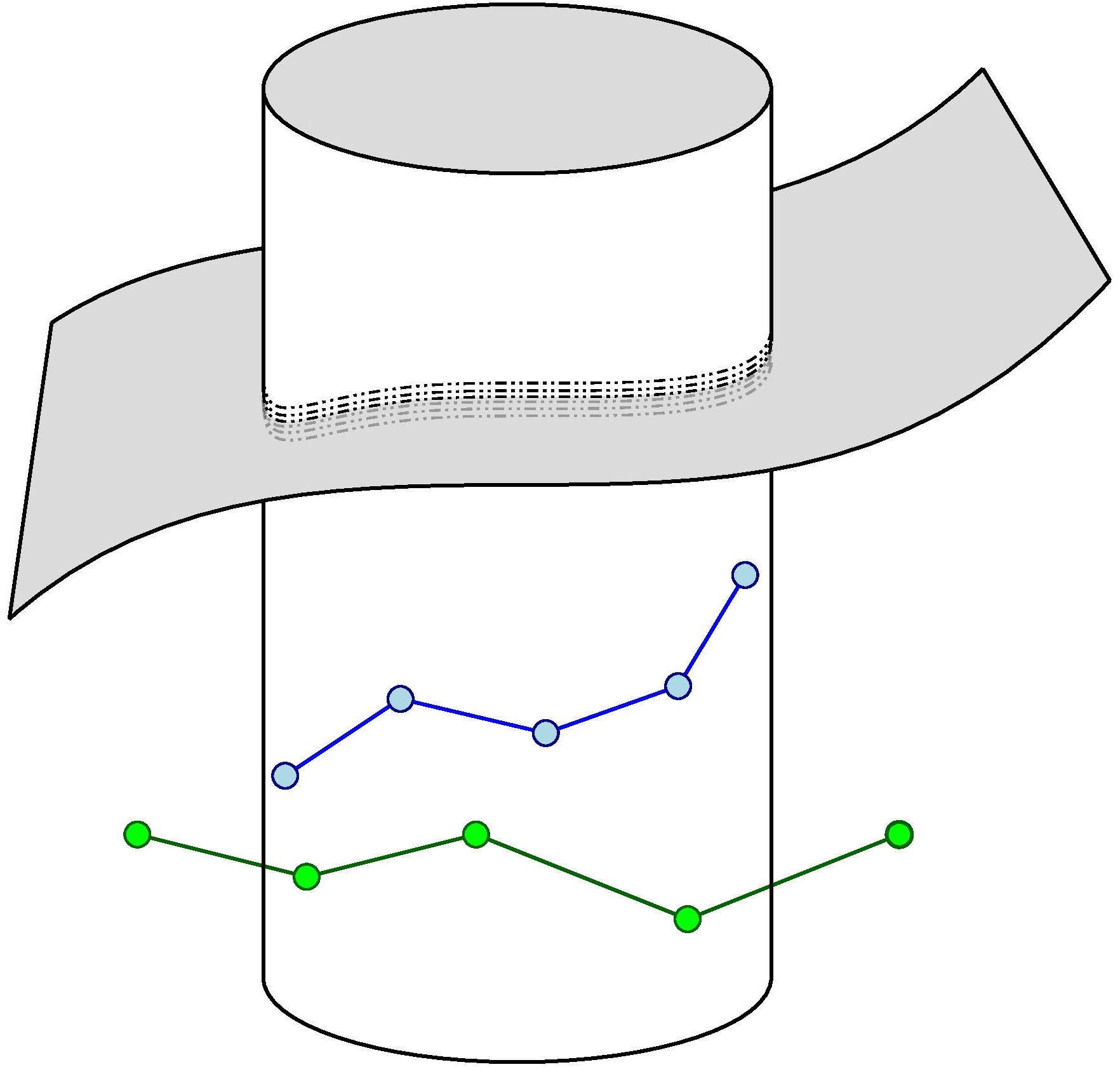}
	\begin{scriptsize}
	\put(11.73,13){\color{black}\line(0,1){5}}
	\put(19,25.5){\color{black}\line(-1,0){3}}
	\put(16,25.5){\color{black}\line(0,1){5}}
	\put(15,34){$p^j$}
	\put(10.5,8.5){$u^j$}
	\put(3,93){$\mathbb{R}^6$}
	\put(26,77){$\Gamma$}
	\put(84,80){$\Sigma$}
	\end{scriptsize}     
  \end{overpic}
   \hspace*{0mm}
  \begin{overpic}[height=51 mm]{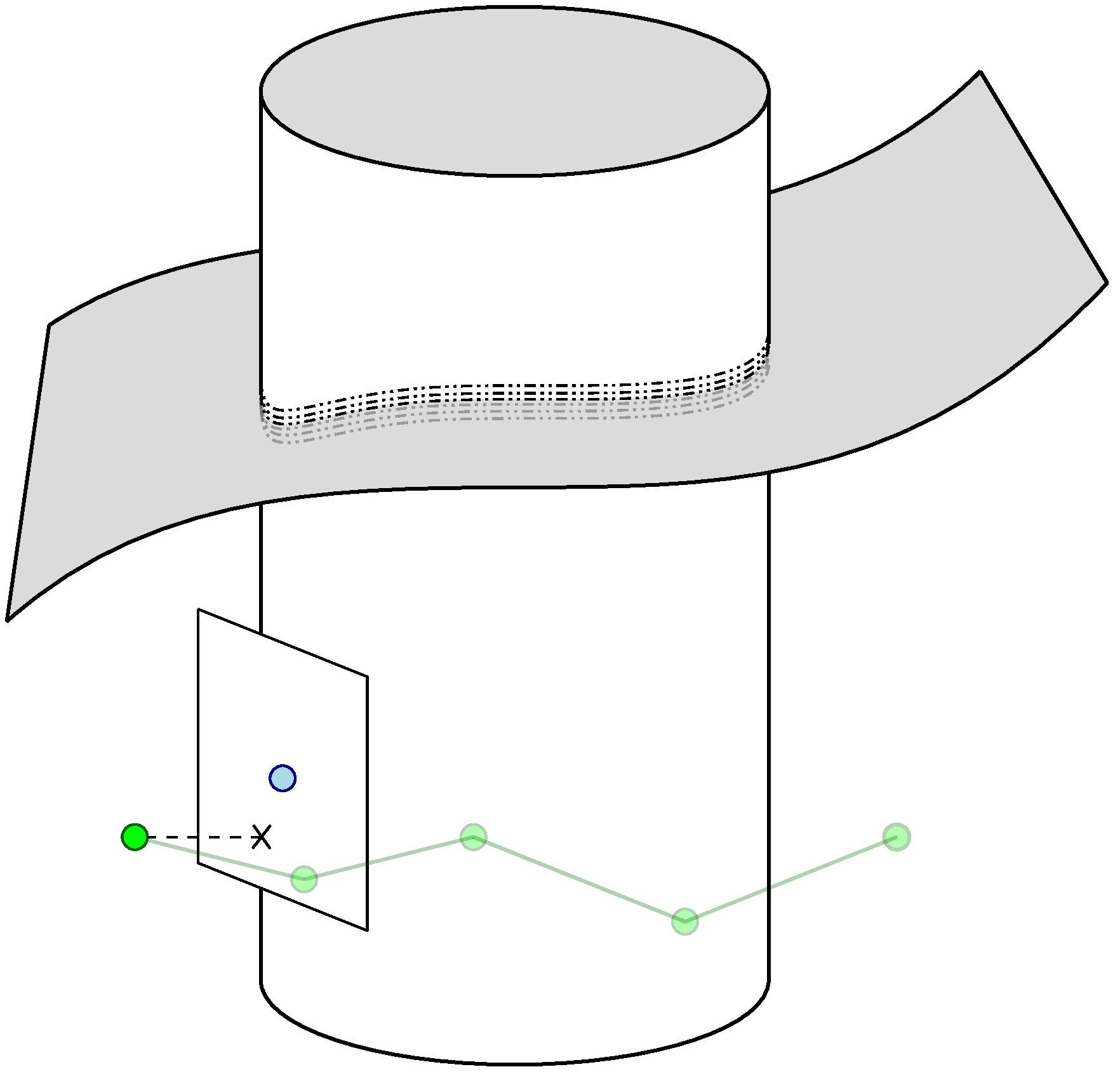}
   \begin{scriptsize}
    \put(23,37){\color{black}\line(1,0){15}}
    \put(3,93){$\mathbb{R}^6$}
	\put(26,77){$\Gamma$}
	\put(84,80){$\Sigma$}
	\put(40,37){$T_{p^j}\left(\Gamma\right)$}
	\end{scriptsize}   
  \end{overpic}
   \hspace*{0mm}
  \begin{overpic}[height=51 mm]{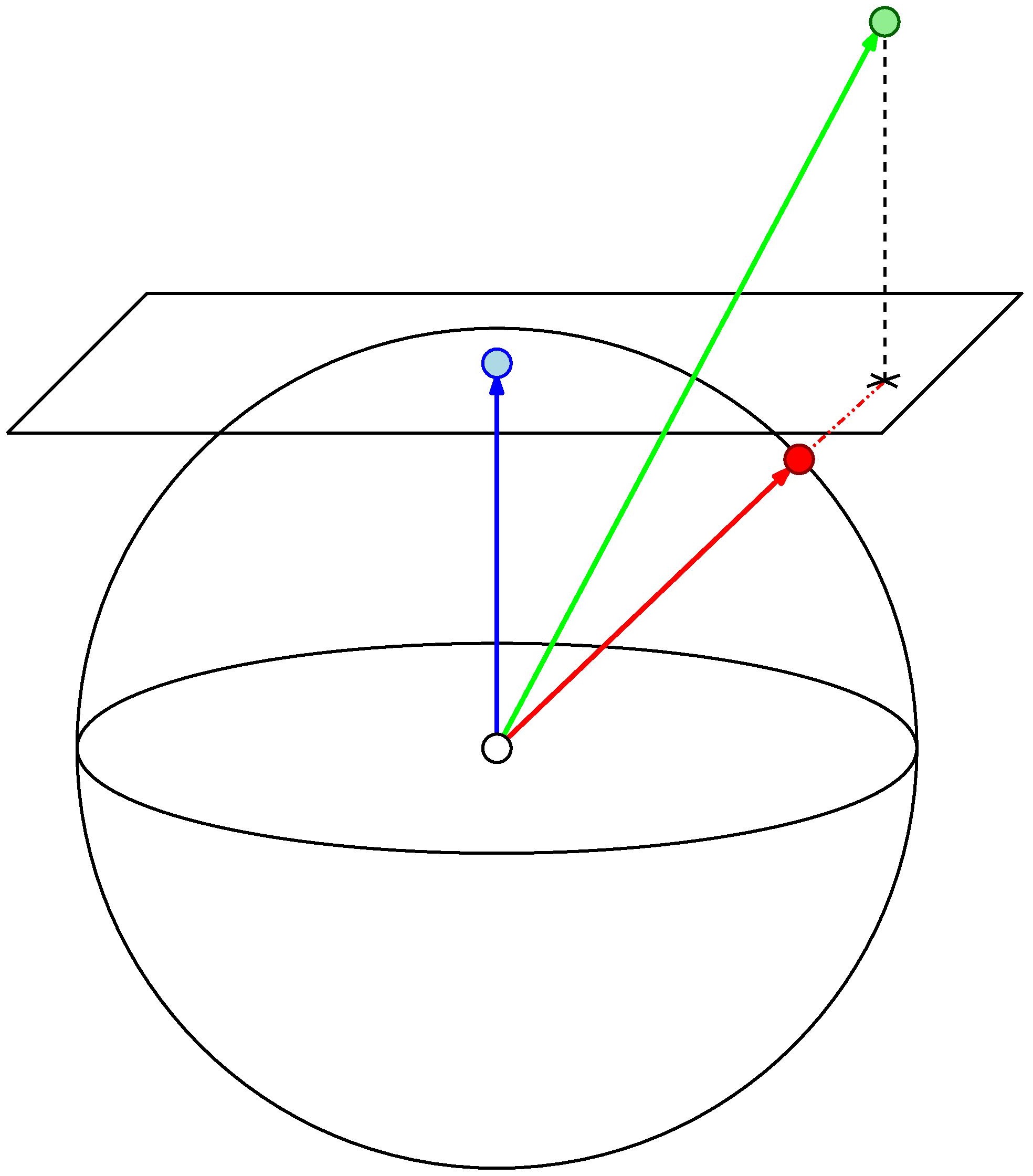}
	\begin{scriptsize}
    \put(34,84){\color{black}\line(0,-1){11}}
    \put(3,95){$\mathbb{R}^3$}
    \put(18,91){$T_{\left(p_1^j,p_2^j,p_3^j\right)}\left(S^2\right)$}
	\end{scriptsize}     
  \end{overpic}
	\caption{Imaginative illustration of the ``generalized cylinder" $\Gamma\subset\mathbb{R}^6$. Left: The blue curve indicates the initial path while the green one stands as the updated curve (not necessarily belonging to $\Gamma$). The red dots are in fact the possible imaginative pedal points on the singularity surface $\Sigma$. Middle: The projection of the updated curve into the the tangent space. Right: A ``realistic" interpretation of such projections in the orientation space. The blue vector shows the original orientation of the pose $p^j$ and the green vector depicts the orientation after update. The cross, $\mathsf{x}$, stands as the back-projection into the ``tangent to sphere at $p^j$", while the red vector describes the final accepted motion after back-projection from the tangent to sphere. }
	\label{fig:tangent:project}
\end{center}
\end{figure}   
\begin{lemma}  
Assume $\hat{p} \in \Gamma^{n-2} \subset \mathbb{R}^{6(n-2)}$. 
Then
\begin{equation}
\left(\nabla\mathcal{C}|_{\Gamma^{\,(n-2)}}\right)_{\hat{p}} =
\left(  \begin{array}{cccc}\,
\left({\mathcal{U}}^{\ 2}_{\ 6\times1}\,\right)^{T},\, 
\left({\mathcal{U}}^{\ 3}_{\ 6\times1}\,\right)^{T},\,
\cdots ,\,
\left({\mathcal{U}}^{\ n-1}_{\ 6\times1}\,\right)^{T}\, 
\end{array}\right)^{T}_{6 (n-2) \times 1} - {\hat{p}},
\label{eq:proj:nabla}
\end{equation}  
where $\forall i,\,\,2\leq i\leq n-1$ we have
\begin{equation}
\left({\mathcal{U}}^{\ i}_{\ 6\times1}\right)^{T} = \left( \mathrm{Pr}_{(p_1^i,p_2^i,p_3^i)}\left( u_1^i,u_2^i,u_3^i\right),u_4^i,u_5^i,u_6^i\right)^T,
\end{equation}
where $\mathrm{Pr}_{\left( p_1^i,p_2^i,p_3^i\right)}:\mathbb{R}^3\longrightarrow T_{\left( p_1^i,p_2^i,p_3^i\right)}\left( S^2\right)$ is the orthogonal projection onto the sphere's tangent plane at $\left( p_1^i,p_2^i,p_3^i\right)$.
\label{lemma:projection}
\end{lemma}
\begin{proof}
$\Gamma$ is a regular submanifold of $\mathbb{R}^6$. This implies $\Gamma^{n-2}$ being a regular submanifold of $\mathbb{R}^{\,6(n-2)}$. Now, using Lemma\,\ref{lemma:oneil} we have $\left(\nabla\mathcal{C}|_{\Gamma^{\,(n-2)}}\right)_{\hat{p}} = \mathrm{Pr}_{\hat{p}}\left(\nabla\mathcal{C}\right)_{\hat{p}} = \mathrm{Pr}_{\hat{p}}\left(\hat{u}-\hat{p}\right)$ where $\mathrm{Pr}_{\hat{p}}:\mathbb{R}^{\,6(n-2)}\longrightarrow T_{\hat{p}}\left(\Gamma^{n-2}\right)$ is the orthogonal projection into the tangent of $\Gamma^{n-2}$ at $\hat{p}$. This results in
\begin{equation}
\left(  \begin{array}{cccc}\,
\left({\mathrm{Pr}_{p^2}\left( u^2 \right)}\,\right)^{T},\, 
\left({\mathrm{Pr}_{p^3}\left( u^3 \right)}\,\right)^{T},\,
\cdots ,\,
\left({\mathrm{Pr}_{p^{n-1}}\left( u^{n-1} \right)}\,\right)^{T}\, 
\end{array}\right)^{T}_{6 (n-2) \times 1} - {\hat{p}},
\label{eq:proj:nabla:2}
\end{equation}
where $\mathrm{Pr}_{p^i}:\mathbb{R}^6\longrightarrow T_{p^i}\left(\Gamma\right)$ is the orthogonal projection onto the tangent space at $p^i\in\Gamma$. Now having in mind that:
\begin{eqnarray}
T_{\hat{p}}\left(\Gamma^{n-2}\right)\cong
T_{p^2}\left(\Gamma\right)\times T_{p^3}\left(\Gamma\right)\times \ldots T_{p^{n-1}}\left(\Gamma\right)\cong\\
T_{(p_1^2,\,p_2^2,\,p_3^2)}\left(S^2\right)\times T_{(p_1^3,\,p_2^3,\,p_3^3)}\left(S^2\right)\times \ldots T_{(p_1^{n-1},\,p_2^{n-1},\,p_3^{n-1})}\left(S^2\right)\times\mathbb{R}^{3(n-2)},
\label{tangent:cylinder} 
\end{eqnarray}
one deduces that orthogonal projections of Eq.\,\ref{eq:proj:nabla:2} are merely the orthogonal projection on the sphere and hence Eq.\,\ref{eq:proj:nabla} is fulfilled (Fig.\,\ref{fig:tangent:project} illustrates the steps of the proof). 
\end{proof}
\subsubsection{Projection from Cylinder's Tangents to the Cylinder}
In the final step we desire the projection of the updated piecewise smooth curve from the cylinder's tangents to the cylinder. Such projection is easily done through normalization of the orientation variables, namely by the following map (cf Fig.\,\ref{fig:tangent:project}-right):
\begin{equation*}
\left( u_1, u_2, u_3, u_4, u_5, u_6 \right) \longmapsto \left( \frac{u_1}{\| u_1^2+u_2^2+u_3^2\|}, \frac{u_2}{\| u_1^2+u_2^2+u_3^2\|}, \frac{u_3}{\| u_1^2+u_2^2+u_3^2\|}, u_4, u_5, u_6 \right).
\end{equation*}
\begin{figure}[t!] 
\begin{center}   
  \begin{overpic}[height=30mm]{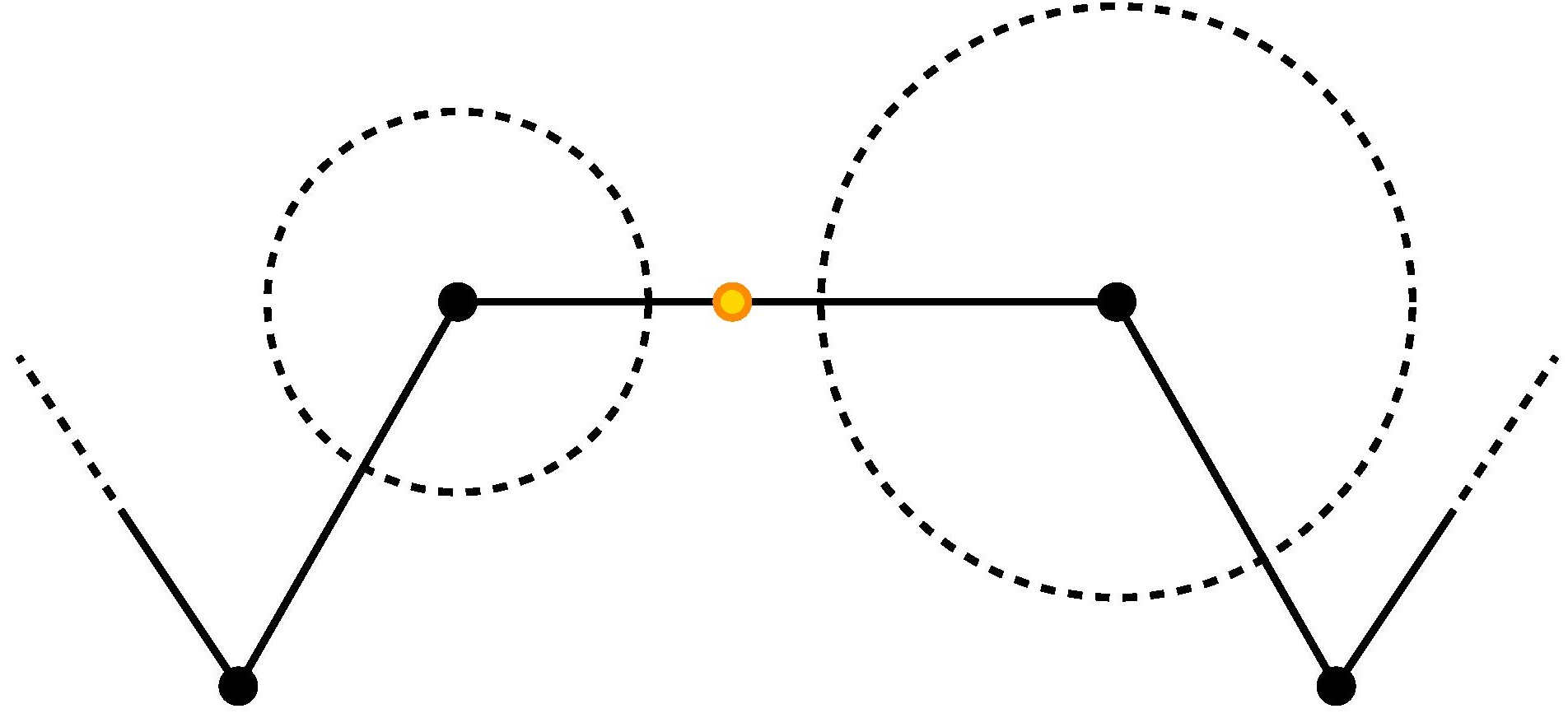}
  \begin{small}
  \put(28,31){$p^i$}
  \put(70,31){$p^{i+1}$}
  \put(1,38){$N_{r^i}\left(p^i\right)$}
  \put(40,0){$N_{r^{i+1}}\left(p^{i+1}\right)$}
  \end{small}     
  \end{overpic}
  \hspace*{5mm}
  \begin{overpic}[height=30mm]{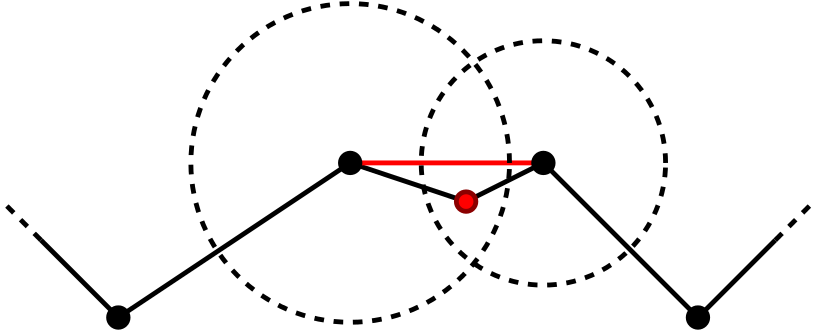}
  \begin{small}
  \put(40,25){$p^{i-1}$}
  \put(65,25){$p^{i+1}$}
  \put(55.5,-2){$p^i$}
  \put(0,35){$N_{r^{i-1}}\left(p^{i-1}\right)$}
  \put(77,35){$N_{r^{i+1}}\left(p^{i+1}\right)$}
  \put(57,13){\color{black}\line(0,-1){11}}
  \end{small}     
  \end{overpic}
	\caption{Imaginative illustration of minimal singularity-free cover of the curve. Left: If a segment is not fully covered by singularity-free balls then a breakpoint will be added in the mid of the uncovered part. Right: If two adjacent singularity-free balls are covering a breakpoint that breakpoint is removable. Note that such a removal should be in such a way that it does not alter the cover for adjacent to the neighboring breakpoints.}
	\label{fig:exclusion:inclusion}
\end{center}
\end{figure}
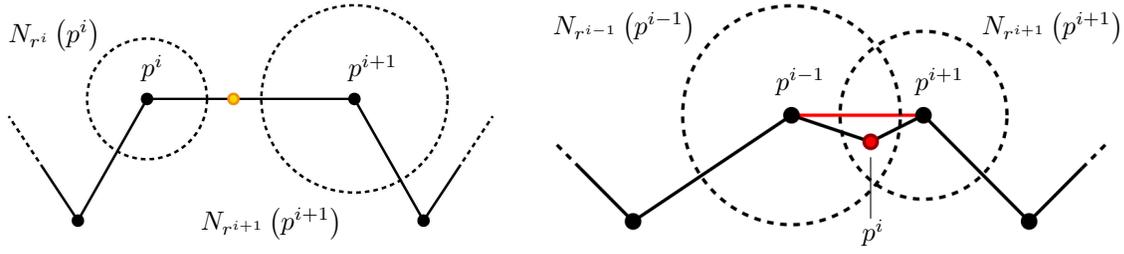   
\subsection{Finite Singularity-free Cover}
\label{sec:finite:cover}
The obtained variation of the initial path is compact in $\mathbb{R}^6$ and hence has a finite open cover. This leads one to the idea of a finite singularity-free cover. Such a cover is consisted of \emph{singularity-free balls} which are extensively discussed in \cite{rasoulzadeh2019linear2}. In this context we approach the concepts of \emph{inclusion} and \emph{exclusion} of breakpoints. In the following discussions, $r^i$ stands for the distance to the closest pedal point with respect to breakpoint $p^i$, guaranteeing the existence of the singularity-free ball $N_{r^i}\left(p^i\right)$.
\subsubsection{Inclusion of Breakpoints}  
For each two consecutive breakpoints, namely $p^i$ and $p^{i-1}$, the algorithm checks whether singularity-free balls $N_{r^i}\left(p^i\right)$ and $N_{r^{i-1}}\left(p^{i-1}\right)$ cover the segment in between. If the segment is not fully covered the algorithm implements a new breakpoint in exactly the midway of uncovered part of the segment as depicted in Fig.\,\ref{fig:exclusion:inclusion}-left.
\subsubsection{Exclusion of Breakpoints}
While the previous step assures one of singularity freeness of the curve, it is adequate not to have excessive included breakpoints. The initial idea is to exclude a breakpoint $p^i$ if it is covered by the adjacent balls $N_{r^{i-1}}\left(p^{i-1}\right)$ and $N_{r^{i+1}}\left(p^{i+1}\right)$ (cf.\ Fig.\ref{fig:exclusion:inclusion}-right). However, one should be mindful of the fact that the breakpoint $p^i$ itself creates the ball $N_{r^i}\left(p^i\right)$ which along with $N_{r^{i+2}}\left(p^{i+2}\right)$ play the role of the cover for $p^{i+1}$. Deleting $p^{i}$ then may cause the removal of its corresponding ball and hence a blown possible cover for $p^{i+1}$. In order to avoid such an undesirable situation the algorithm first labels the breakpoints which are doubly covered by adjacent balls (a point $p^i$ is doubly covered if it belongs to  $N_{r^{i-1}}\left(p^{i-1} \right) \cap N_{r^{i+1}}\left(p^{i+1}\right)$, see Fig.\,\ref{fig:exclusion:inclusion}-right). Then the algorithm groups such breakpoints into \emph{packs} based on the number of consequent occurrence of double covers. Calling the number of breakpoints in each pack the \emph{size} of the pack, the algorithm takes the following strategy:
\vspace*{1mm}
\begin{itemize}
\item[$\bullet$] Packs of size 1 are deleted,
\item[$\bullet$] In the packs of size bigger than 1 the breakpoints with odd numbers are removed. 
\end{itemize}
\vspace*{1mm}
Finally, the algorithm repeats these steps till no pack can be found \footnote{In practice, due to computational reasons the algorithm is set to keep 6 number of breakpoints at least.}. Naturally, in the above process the start- and end-pose are entirely neglected. 
%
%
%
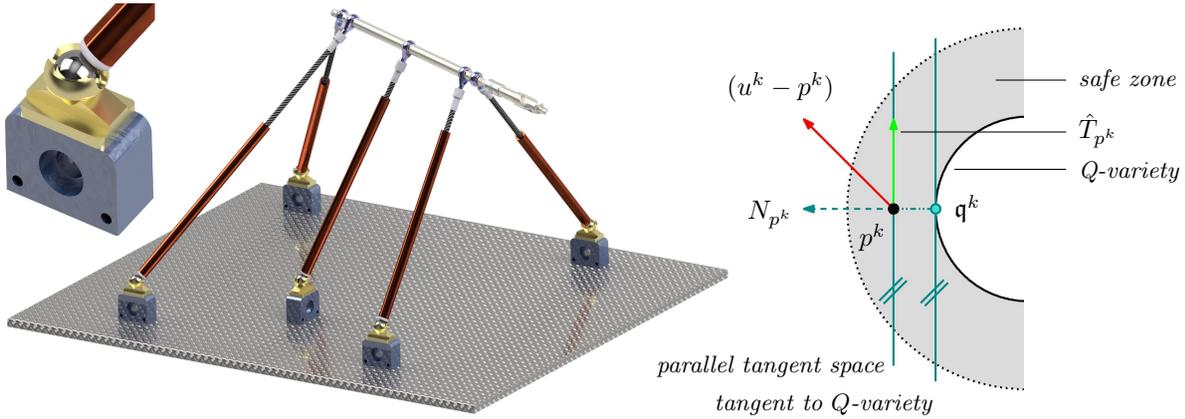
\begin{figure}[t!]
\includegraphics[height = 55 mm, width = 0.622\columnwidth]{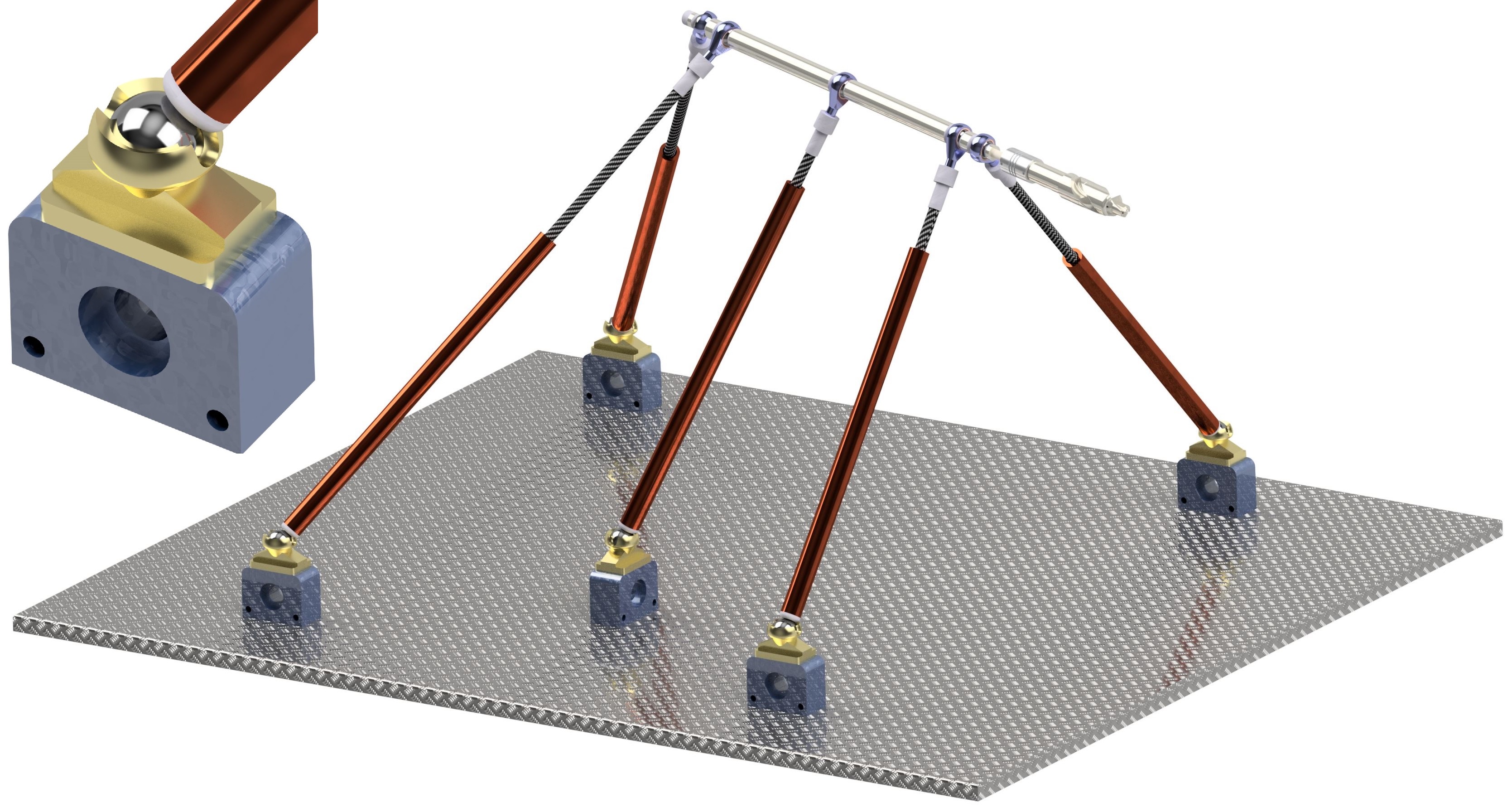}
\hspace*{5mm}
\begin{overpic}[height = 55 mm]{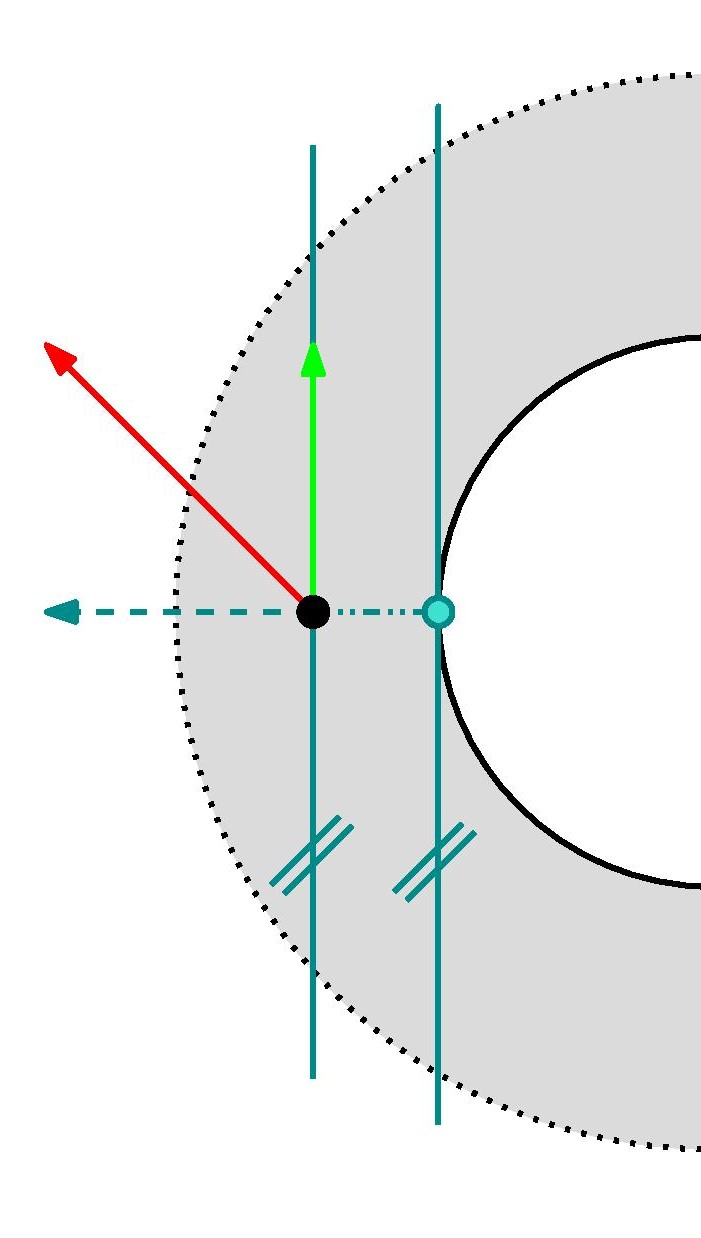}
  \begin{small}
  \put(17,41.5){$p^k$}
  \put(40,48.5){$\mathfrak{q}^k$}
  \put(-10,48){$N_{p^k}$}
  \put(-32,11){\footnotesize \it parallel tangent space}
  \put(-18,2){\footnotesize \it tangent to $Q$-variety}
  \put(-15,78){$(u^k - p^k)$}
  \put(70,67.5){$\hat{T}_{p^k}$}
  \put(70,58){\footnotesize \it Q-variety}
  \put(70,80){\footnotesize \it safe zone}
  \put(50,81.5){\color{black}\line(1,0){17}}
  \put(27,69){\color{black}\line(1,0){40}}
  \put(40,60){\color{black}\line(1,0){27}}
  \end{small}     
\end{overpic}     
\label{fig:design}
\caption{Left: Illustration of a simple pentapod with the LP-property in addition to a magnified view of the base spherical joints. Base joint limit is modelled by a cone with an axis orthogonal to the base plane. Right: Imaginative illustration of the projection on a tangent of a $Q$-variety.}
\end{figure}
\section{Joint Analysis} 
\label{sec:joint}
From the computational kinematics point of view, feasibility of optimized motion by the joint and linkages is of utmost importance. These physical limits restrict the end-effector movements and are of two types, namely, \emph{spherical joint limits} and \emph{prismatic joint limits}. It turns out that dealing with these restrictions can be translated into the language of variational path optimization. In fact, as it will be described in the coming subsections, it is possible to think of these restrictions as hypersurfaces in $\mathbb{R}^6$. One last remark before plunging into the details, would be regarding the design of the pentapod. In the theorems to come within this section, it is not necessary to restrict ourselves to \emph{simple pentapods} since they are valid for \emph{general linear pentapods} in which the base anchors belong to $\mathbb{R}^3$. Hence, the findings in this section are more general and can be considered for a more general framework as well.
\vspace*{3mm}
\subsection{Prismatic Joint Limits} In the position space a point on the prismatic extendible part can be located within two spheres with radii $\rho_{min}$ and $\rho_{max}$. Considering this setting in pose space $\mathbb{R}^6$ results the following theorem:
\begin{theorem}
Assume a ``general linear pentapod" with minimum possible prismatic length $\rho_{min} > 0$ and maximum $\rho_{max}$, then the set of all feasible poses with $\rho_{min}\leq \rho\leq \rho_{max}$ forms a ``smooth hyperquadric" in $\mathbb{R}^6$.  
\label{prismatic:limit}
\end{theorem}
\begin{proof}
Writing down the equations of the sphere with the radius $\rho_{min}\leq\rho\leq\rho_{max}$ around a base spherical joint with coordinates $(x_i,y_i,z_i)^{T}$ gives
\begin{equation}
f_p:=\left( r_{{i}}\,u_{{1}}+u_{{4}}-x_{{i}} \right) ^{2}+ \left( r_{{i}}\,u_{
{2}}+u_{{5}}-y_{{i}} \right) ^{2}+ \left( r_{{i}}\,u_{{3}}+u_{{6}}-z_{{i
}} \right) ^{2}-{\rho ^{2}}
\label{prismatic:limit}
\end{equation}
once again by considering $f_p\in\mathbb{C}\left[u_1,u_2,u_3,u_4,u_5,u_6\right]$, Eq.\,\ref{prismatic:limit} describes a quadric hypersurface in pose space $\mathbb{R}^6$. Calculating $\nabla f_{p} = 0$, one observes that it vanishes on $\mathbf{V}\left(f_{p}\right)$ iff $\rho = 0$. Hence $\mathbf{V}\left(f_{p}\right)$ is a smooth algebraic variety.
\end{proof}
\subsection{Base Spherical Joint Limits} 
\begin{theorem}
Assume a ``planar" pentapod with similar base spherical joints as depicted in Fig.\,\ref{fig:design}-left, then the set of all poses resulting in the maximum freedom of movement of the $i$-th leg with respect to the $i$-th base joint limit forms a ``hyperquadric" in $\mathbb{R}^6$. 
\label{theorem:basecones}
\end{theorem} 
\begin{proof}
The base spherical joints located on the plane allow motions inside a cone of revolution with vertical axis. The implicit equation of such a cone with \emph{apix angle} $\theta$ and base anchor coordinates $(x_i,y_i,0)^T$ is:
\begin{equation}
f_{bc} = \left({u_6 + r_i u_3}\right)^{2} - \cot ^2 \left( \frac{\theta_{{i}}}{2} \right) \left[ \left( u_4 + r_i u_1 - x_{i} \right) ^{2} + \left( u_5 + r_i u_2 - y_{{i}} \right) ^{2}\right].
\label{base:cone}
\end{equation}
One may think of Eq.\,\ref{base:cone} as a polynomial in $\mathbb{C}\left[u_1,u_2,u_3,u_4,u_5,u_6\right]$. Hence, the $i$-th base cone with an apix limit of $\theta_i$ forms a hyperquadric in $\mathbb{R}^6$. Calculating $\nabla f_{bc} = 0$, one observes that it vanishes on $\mathbf{V}\left(f_{bc}\right)$ iff the platform anchor point collapse with its corresponding base anchor point. It is noteworthy that due to the fact that $\rho_{min} \neq 0$ such a collapse is unreachable.
\end{proof}
Now, the idea is to implement new terms into the cost function in such a way that if a pentapod's leg is at
\vspace*{1mm}
\begin{itemize}
\item[$\bullet$] its maximum/minimum extent,
\item[$\bullet$] its maximum angular limit with respect to base cones,
\end{itemize}
\vspace*{1mm}
the algorithm allows a slide of the end-effector at the maximum/minimum extent or angular limit. Geometrically, this means that if a breakpoint $p^k$ is located on one of the hyperquadrics (from now on called $Q$-variety), mentioned in Theorems\,\ref{prismatic:limit}/\ref{theorem:basecones}, in such a way that if the update of $p^k$ is going out of the Q-variety then the algorithm permits an update of it on $T_{p^k}\left( Q\right)$. Naturally, due to practical reasons, the algorithm must consider such breakpoint already located on $Q$ if the breakpoint is closer than a certain positive real value $\epsilon$ to it (cf. Fig.\,\ref{fig:design}-right). Hence in such a terminology the projection on the tangent will happen only if $\langle u^k - p^k, p^k - \mathfrak{q}^k \rangle < 0$, where $\mathfrak{q}^k$ is the corresponding closest pedal to $p^k$ on Q-variety.
\begin{definition} 
Assume a breakpoint $p^k$ of a piecewise smooth path in $\mathbb{R}^6$ and a $Q$-variety, as mentioned in Theorems\,\ref{prismatic:limit}/\ref{theorem:basecones}, such that $\mathfrak{d}\left(p^k, Q\right) <\epsilon$ where $\epsilon\in\mathbb{R}^{+}$. If $u^{k}$ is the corresponding point to $p^k$ after variation and if $\mathfrak{q}^{k}$ is the orthogonal projection of $p^k$ on $Q$ then define
\begin{equation}
\hat{T}_{p^k} := (u^k - p^k) - \langle (u^k - p^k) , N_{p^k}\rangle\,N_{p^k},
\label{eq:tangenttoQ}
\end{equation}
where $N_{p^k}:= \left(p^k-\mathfrak{q}^{k}\right)/\|p^k-\mathfrak{q}^{k}\|$. Then the $k$-th breakpoint's update should be replaced by $\hat{T}_{p^k}$. 
\label{def:jointanalysis}
\end{definition} 
\begin{remark}
\label{remark:intersection11}
In the variational path optimization process it often happens that a breakpoint $p^k$ breaches the $\epsilon$ vicinity of more than one $Q$-variety (e.g. by almost reaching the maximum extent of more than one leg). In such a situation the vector $(u^k - p^k)$ should be projected on the intersection of the corresponding tangent spaces. 
\end{remark}
\section{Variational Path Optimization Algorithm's Manual}
\label{sec:algorithmdetails}
This section represents a short manual of using the variational path optimization algorithm. The manual is accompanied by an algorithm's flowchart given in Appendix\,\ref{appendix:C}.\\
\subsection{Input}
\label{sec:input}
The algorithm will ask for two main sets of inputs namely, \emph{design parameters} and \emph{optimization parameters}.
\subsubsection{Design parameters} These variables define the architectural structure of the pentapod and hence its singularity variety. These variables of $\mathbf{r}$, $\mathbf{X}$ and $\mathbf{Y}$ stand as the defining parameters of the base and platform anchor points while the parameters $\alpha$ and $\beta$ describe either the pattern of platform anchor points (in the LP-case) or the co-linearity of certain number of base anchor points (in LO-cases). In a more general terminology, the design parameters inform the algorithm on building up the geometric obstacle (i.e. $\Sigma$-variety).
\subsubsection{Optimization parameters} These variables globally control the variations of the initial curve. These parameters are read as follows:
\vspace*{1mm}
\begin{itemize}
   \item[$\bullet$] $\mathsf{n}$:  the number of breakpoints required for the piecewise smooth curve,
   \item[$\bullet$] $\lambda$:     the geodesic weight, 
   \item[$\bullet$] $\eta$:        the bending weight,
   \item[$\bullet$] $\mathsf{growth}$:  is a percentage value indicating the maximum and minimum permitted variation of the curve at each iteration. Experimentally, it is fixed to be $\pm5\%$.
\end{itemize}  
\subsection{Process I}
\label{process:1}
The main task of \emph{process I} is to create a data collection of the inputs. In fact, these are the calculations that the algorithm needs to do to have an initial understanding of the situation of the initial path and the singularity variety (obstacle). This process is done only \emph{once} per run. During this process the manipulator will not perform any action and hence the \emph{clock} will not \emph{tic}.\\
\begin{itemize}
\item[\textbf{1.}] Retrieving the data on the given singularity-free initial path within the joint limits.\\
\item[\textbf{2, 3.}] The algorithm evaluates the situation of the initial curve with respect to the $\Sigma$-variety (obstacle) by obtaining the corresponding pedal points on the $\Sigma$-variety (obstacle). This act is done in the fashion of the following consecutive steps:
\vspace{1mm} 
\begin{itemize}
\item[$\bullet$] finding the real pedal points (cf. Section\,\ref{section:pedal}), 
\item[$\bullet$] finding distances corresponding to the real pedal points,
\item[$\bullet$] sorting the real pedal points with respect to their distance in such a way that the closest pedal point is labelled as the first.
\end{itemize}
\vspace*{1mm}
\item[\textbf{4.}] The algorithm calculates the minimal singularity-free cover by including/excluding breakpoints (\emph{upon user's request}).
\item[\textbf{5.}] At this point, the algorithm runs a check on joint rate limits. Additionally, the safe zone breaches are flagged (\emph{upon user's request}).
\item[\textbf{6.}] The objective function is evaluated. In the coming iterations one must expect the monotonic descent of the objective function per iteration till convergence at a possible local minimum.
\subsection{Decision I}
\label{decision:1}
Though \emph{Decision I} could be written as a \emph{while loop}, due to imposing more control over the outcomes by the user (i.e. user can call a result before the final convergence) it is presented as a \emph{for loop} (cf.\ Appendix\,\ref{appendix:C}). Additionally, by the start of \textit{Decision I}, the clock \emph{ticks} to measure the \emph{elapsed time}.   
\subsection{Process II}
\label{process:2}
The core of the algorithm is located within the \emph{Process II}.
\vspace*{3mm}
\begin{itemize}
\item[\textbf{1.}] By solving the linear system, the preferred direction $u-p$ is obtained. 
\vspace*{1mm}
\item[\textbf{2.}] The algorithm monitors whether the safe zone of $Q$-varieties is breached or not. If the answer is positive then the update of breakpoints responsible for the breach is done in the fashion of Eq.\,\ref{eq:tangenttoQ} (cf. Remark\,\ref{remark:intersection11}).
\vspace*{1mm}
\item[\textbf{3, 4, 5.}] Computing the step size as explained in Section\,\ref{sec:stepsize}. 
\vspace*{1mm}
\item[\textbf{6.}] Having the step size, the curve's new pose is updated as follows:
\begin{footnotesize}
\begin{lstlisting}[backgroundcolor = \color{lightgray},
                   language = Matlab,
                   xleftmargin = 0cm,
                   framexleftmargin = 1em]
up(2:n-1,:) = p(2:n-1,:) + (step_size) * (u(2:n-1,:));   
\end{lstlisting}
\end{footnotesize}
\vspace*{1mm}
\item[\textbf{7, 8.}] Back-projection into $\Gamma$ as explained in Section\,\ref{sec:projection:cylinder}. 
\vspace*{1mm}
\item[\textbf{9.}] As we project the result into $\Gamma$ there is no guarantee for the minimal singularity-free cover to remain intact. Hence, based on the two algorithms described in Section\,\ref{sec:finite:cover}, the inclusion and exclusion should be redone.
\vspace*{1mm}
\item[\textbf{10.}] The cost function is calculated for the current iteration label $i+1$.
\end{itemize}
\vspace*{3mm}
\end{itemize}
\subsection{Decision II}
\label{decision:2}
The new cost function, $\mathsf{cost\_ function (i+1)}$ is compared with the previous one, namely $\mathsf{cost\_function (i)}$. Since by the gradient descent it is expected of the cost function per iteration to decrease, in \emph{Decision II} a \emph{while loop} is commenced. As long as the cost function is not descending and the value of step size is not \emph{zero} (this can be changed by user to a small number such as $10^{-6}$ for being real-time preservation) the algorithm executes the \emph{while loop}.
\subsection{Process III}
\label{Process:3}
At each iteration of the \emph{while loop} the algorithm halves the step size and reruns \emph{Process II}. This gives the opportunity for the preferred direction to find the proper direction to a local minimum.
\subsection{Output}
\label{output}
The coordinates of the \emph{final curve}, the plot of the optimized motion and the \emph{objective function per iteration diagram} are printed.
\section{Results \& Discussions}
\label{sec:results}
Finally, the results are demonstrated in the form of the following two example. 
\subsection{Example}
\label{sec:example1}
\begin{figure}[t!] 
\begin{center}   
\includegraphics[clip,height = 54 mm]{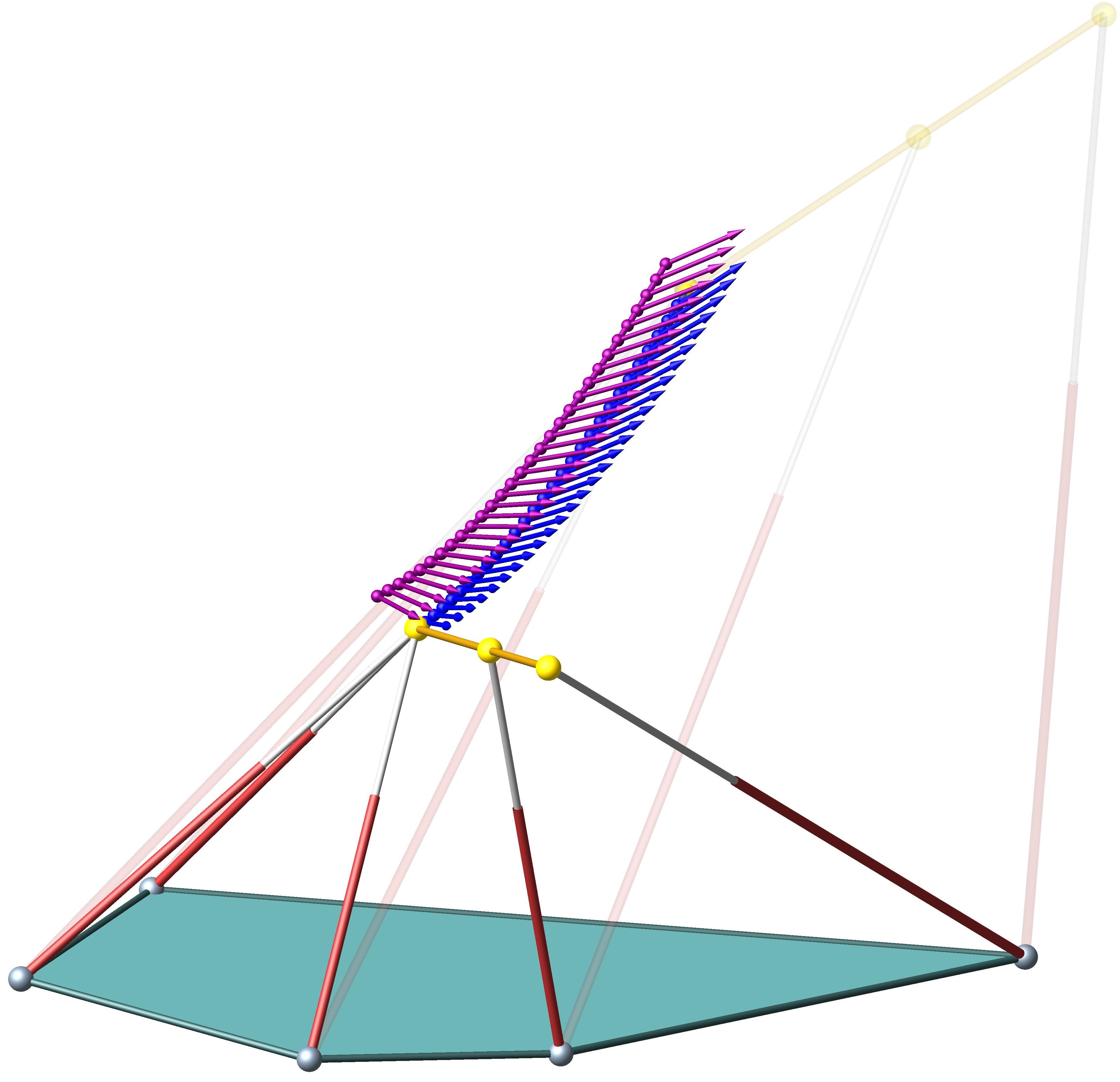}
\hspace*{10 mm}
\includegraphics[height = 54 mm]{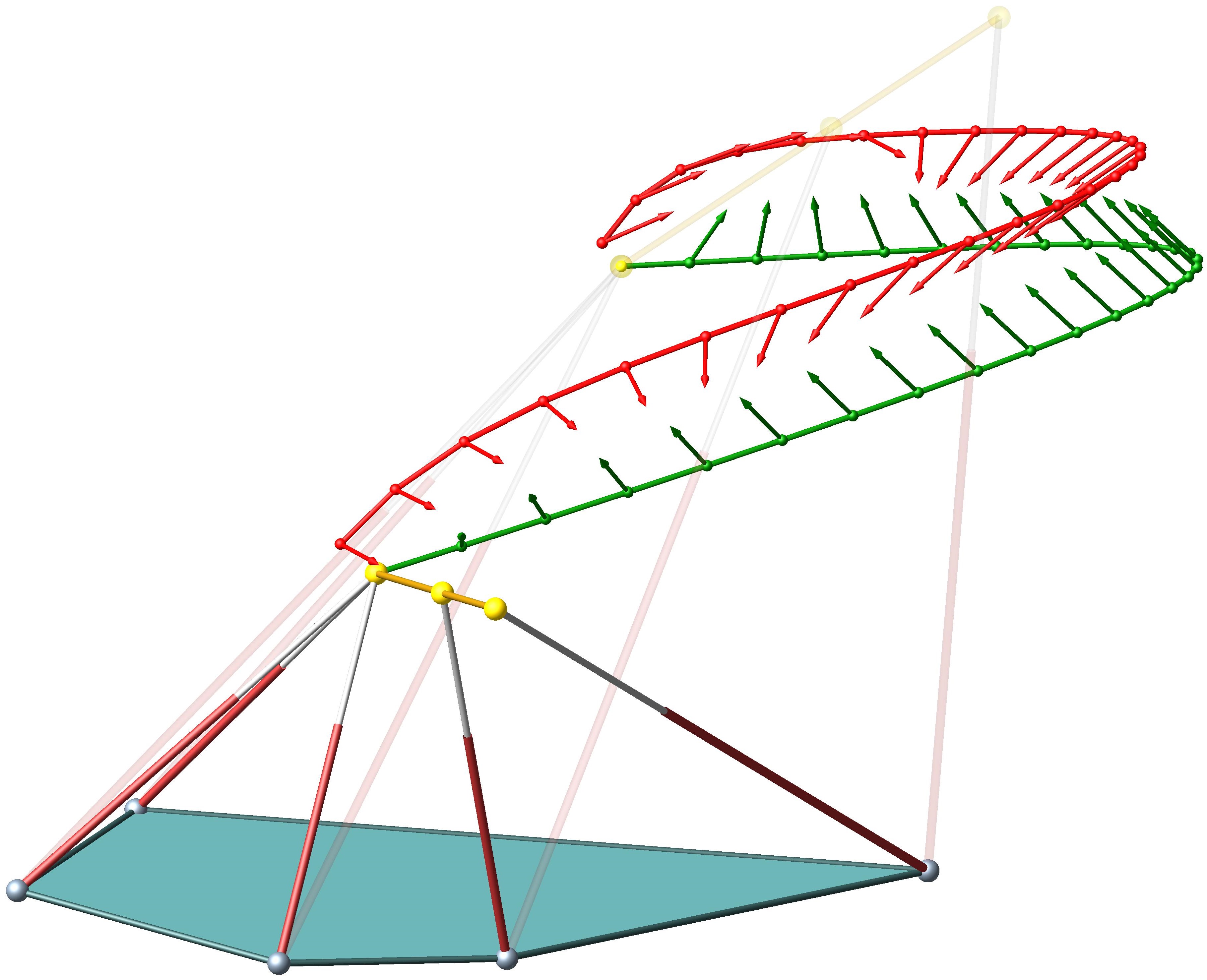}
\vspace*{1 mm}
\begin{overpic}[height = 54 mm]{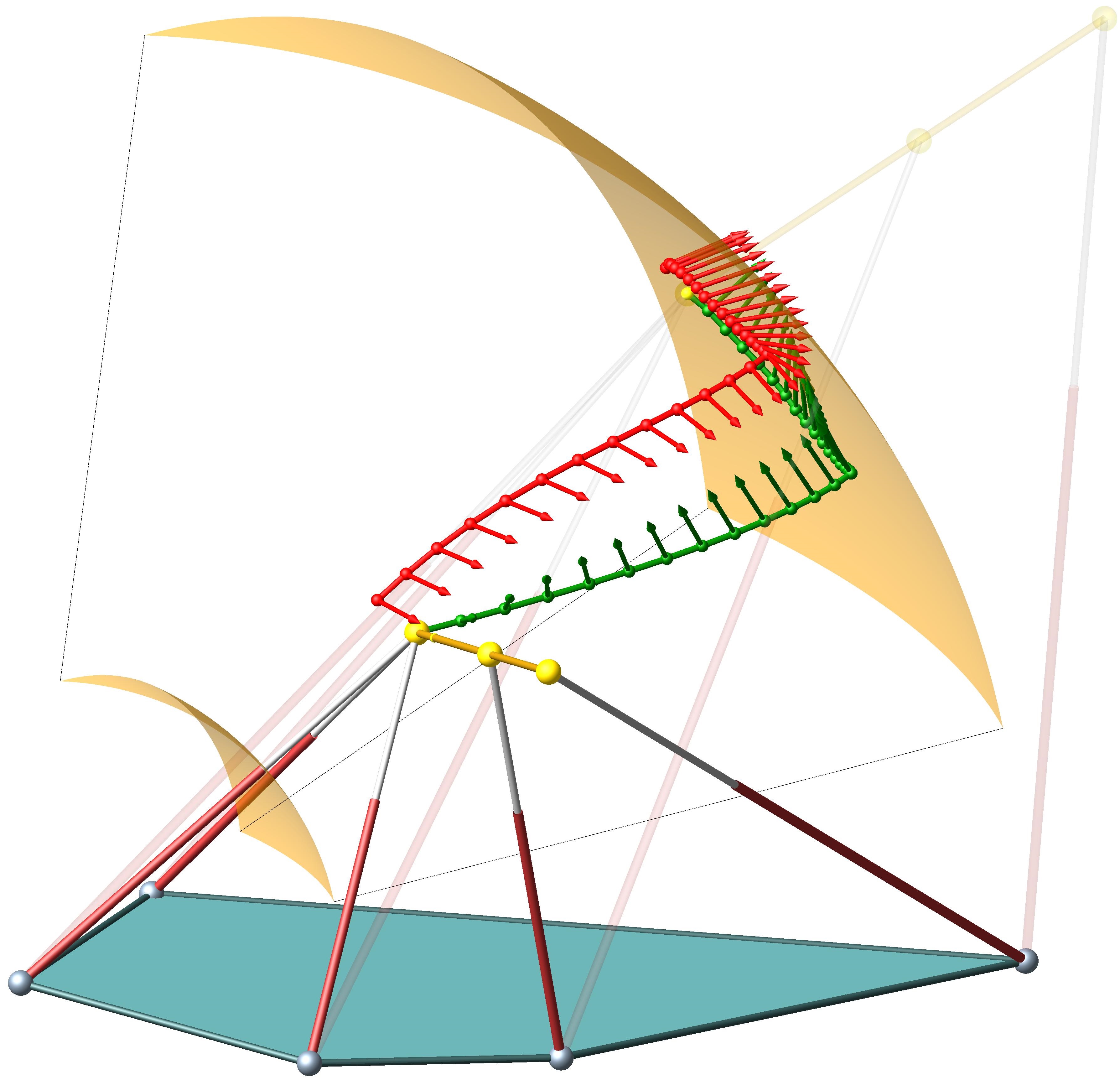}
\begin{small}
\put(-9,20){\footnotesize \it 1st leg}
\end{small}     
\end{overpic}
\hspace*{10 mm}
\begin{overpic}[height = 54 mm]{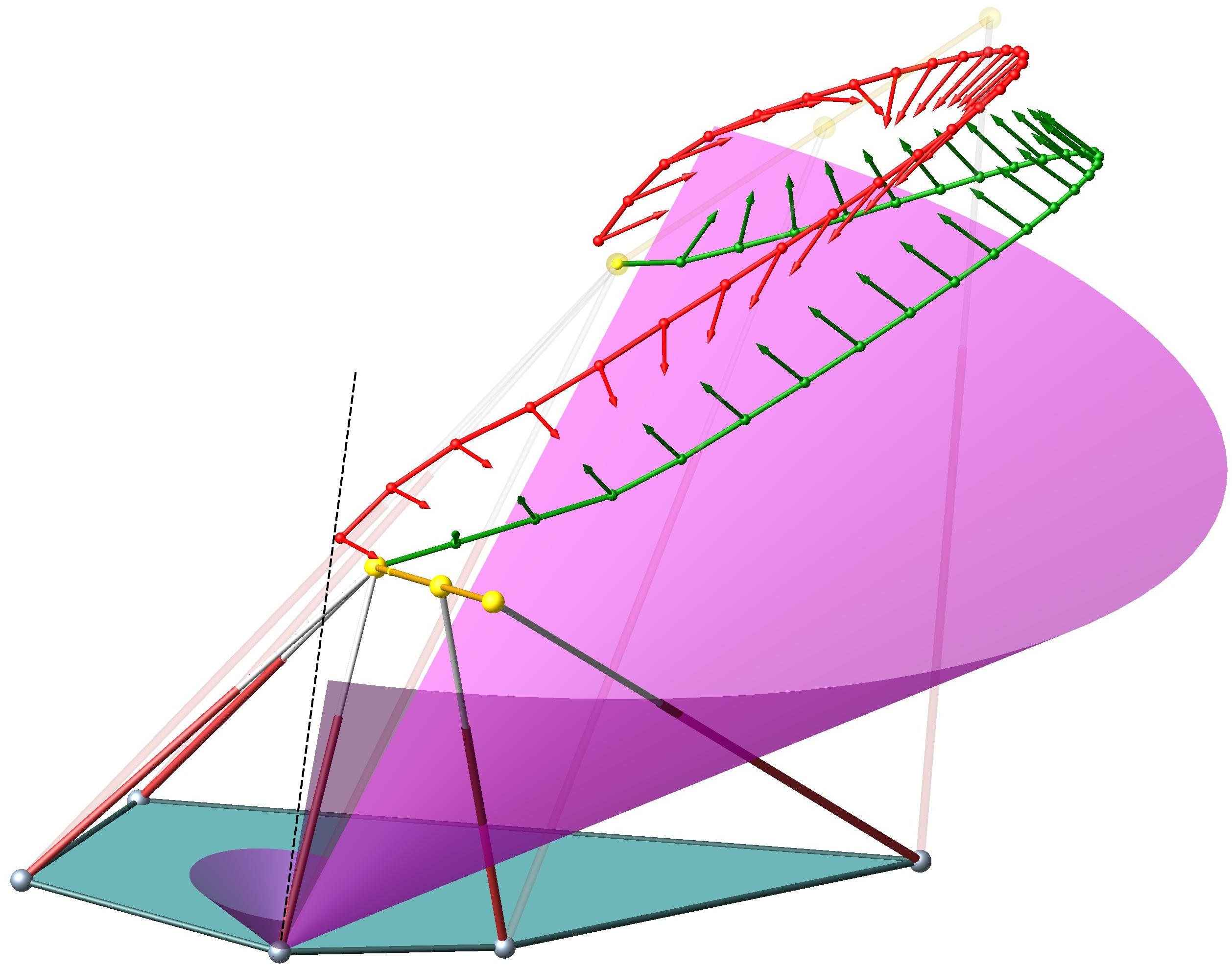}
\begin{small}
%
\put(80,40){\footnotesize \it 2nd cone}
\end{small}     
\end{overpic}
\caption{ Illustration of the paths in $\mathbb{R}^6$ and their corresponding pedals as vector fields along paths in $\mathbb{R}^3$.
Top-left: Initial path (blue) and its corresponding pedals (purple),
Top-right: Final path (green) and its corresponding pedals (red) (\color{blue}blue\color{black}\ objective curve),
Bottom-left: Modification of the final path under 1st prismatic leg's safe-zone breach (\color{cyan}cyan\color{black}\ objective curve), two golden sphere's portions depict the minimum contraction and maximum extension of the first leg,
Bottom-right: Modification of final path under 2nd base cone's safe-zone breach (\color{magenta}magenta\color{black}\ objective curve).
}
\label{fig:result1}
\end{center}
\end{figure}  
%
%
Consider a \emph{simple pentapod} of the 3rd-LO type (cf. Fig.\,\ref{fig:LO}-right) with the following architecture matrix (cf. Eq.\,\ref{BorrasMatrix}):
\begin{equation}
\left( \begin {array}{cccc} 
r_2 &  x_2 &  y_2  & z_2 \\ 
r_3 &  x_3 &  y_3  & z_3 \\ 
r_4 &  x_4 &  y_4  & z_4 \\ 
r_5 &  x_5 &  y_5  & z_5  
\end {array} \right)
=\left( \begin {array}{cccc} 
0 &  5 &  0    & 0 \\ 
0 &  0 &  5    & 0 \\ 
5 &  8 &  3    & 0 \\ 
9 & 12 &  12   & 0  
\end {array} \right).
\label{Arch:num:2}
\end{equation}
Additionally, the second base cone's apex angle (resembling the physical limits of base spherical joints) is set at $108^{\circ}$ while the range of the first leg's prismatic joint varies in the closed interval $\left[ 5.1, 16 \right]$.
The initial singularity-free path is given by the curve
\begin{gather}
\alpha : \left[ 2, 5 \right] \longrightarrow \Gamma \subset \mathbb{R}^6,\\
x \longmapsto\left( \sin\left(\theta\right)\cos\left(\phi\right), \sin\left(\theta\right)\sin\left(\phi\right), \cos\left(\theta\right) , \frac{x + 10}{3}, \frac{x^2 + 10}{3}, \frac{x^3}{30} + 5.333\right),
\end{gather}
where
\begin{equation}
\left(\begin{array}{c}
\theta \\
\phi
\end{array}\right)
=
\frac{5 - x}{3}\,
\left(\begin{array}{c}
0.4\,\pi \\
6.8\,\pi
\end{array}\right)
+
\frac{x - 2}{3}\,
\left(\begin{array}{c}
0.25\,\pi \\
2\,\pi
\end{array}\right).
\end{equation}
The illustrated cases are computed for $30$ breakpoints with the geodesic weight $\lambda = 0.001$ and the bending weight $\eta = 0.05$. Furthermore, the curve's growth is given by $\mathsf{growth} = 5\%$ and safe zone's vicinity is given by $\epsilon = 0.4$ (cf. Definition\,\ref{def:jointanalysis}).\\
Finally, the development of the motion curve are depicted in Fig.\,\ref{fig:result1} and the geometric details of the final optimized motion is given in Table\,\ref{table:results} while Fig.\,\ref{fig:graf}-left shows the descent of the related objective functions per iteration. 
%
\begin{figure}[h!] 
\begin{center}   
  \begin{overpic}[height=50mm, width = 75 mm]{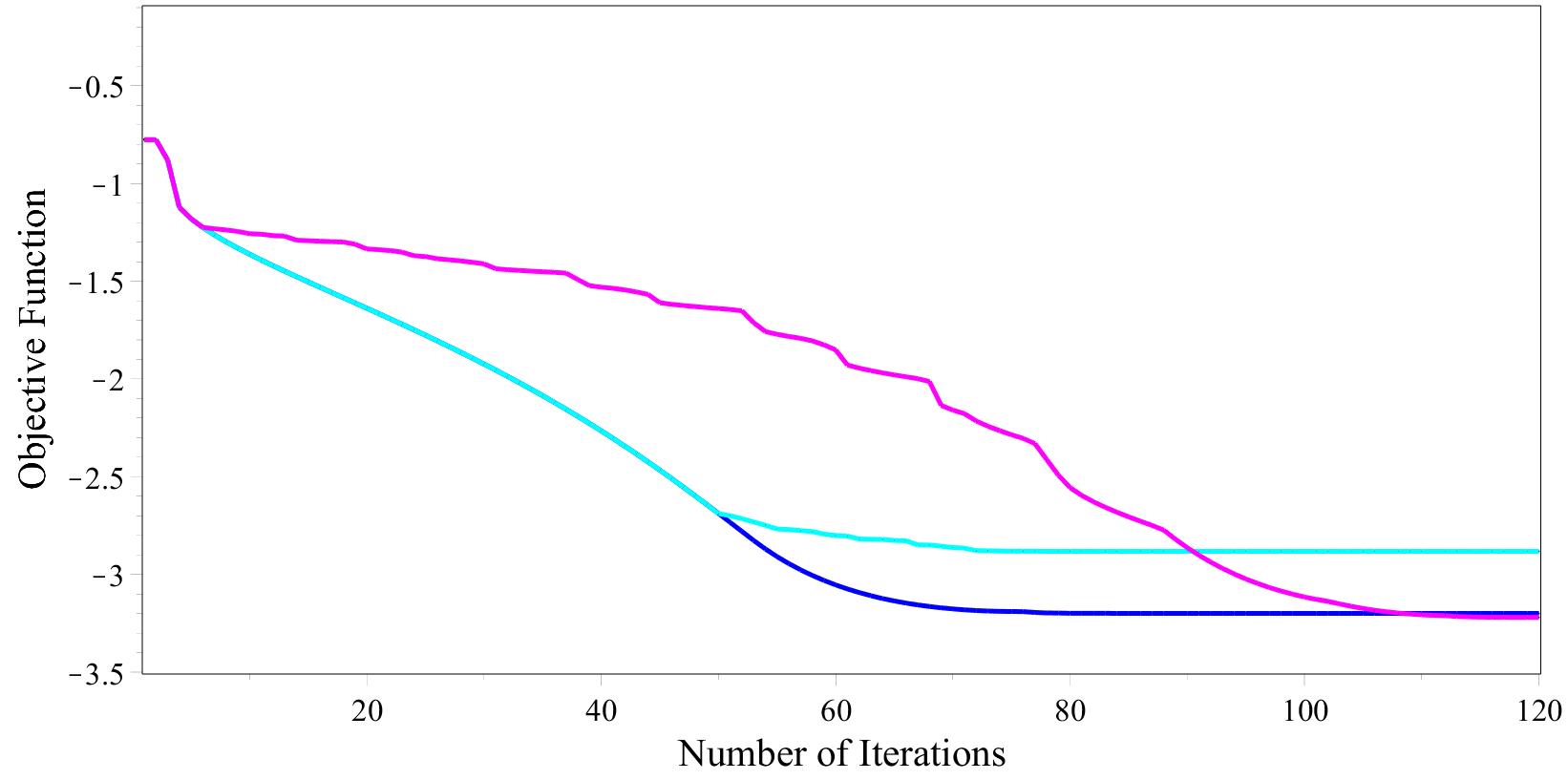}
  \begin{small}
  %
  \end{small}     
  \end{overpic}
  \hspace*{0mm}
  \begin{overpic}[height=50mm, width = 75 mm]{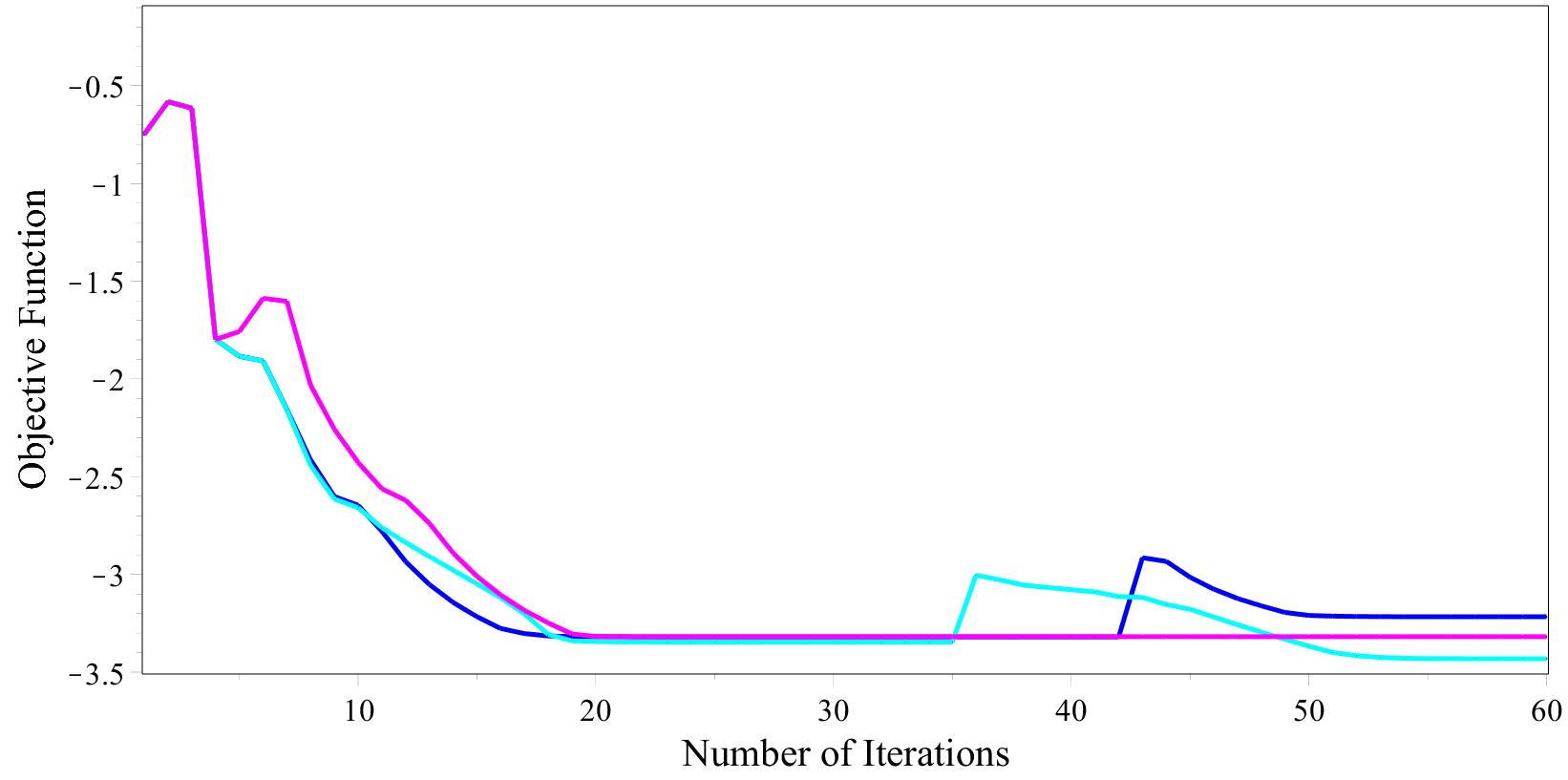}
  \begin{small}
  %
  \end{small}     
  \end{overpic}
  \caption{Illustration of the objective curves corresponding to motions in Fig.\,\ref{fig:result1} (left) and Fig.\,\ref{fig:ex} (right). Note that the sharp changes of the objective functions are due to the inclusion/exclusion of breakpoints (right). }
	\label{fig:graf}
\end{center}
\end{figure}   
\begin{table}[h!]
\centering
\begin{small}
{
 \begin{tabular}{||c |c| c| c| c||} 
 \hline 
                     & Top left  & Top right   & Bottom left   & Bottom right   \\ [0.5ex]
 \hline\hline
 Length              & 11.1822   & 26.5303     & 14.3375       & 23.8307        \\ [0.5ex]
 \hline
 Total curvature     & 0.2379    & 5.3926      & 1.9030        & 4.8903         \\ [0.5ex]
 \hline
 Elapsed time        & -         & 1.6314 s    & 3.9546 s      & 6.7991 s     \\ [0.5ex]
 \hline
\end{tabular}} \\[1ex]
\end{small}
\vspace{0.5	mm}
\caption{Length and total curvature are computed for the final curve in $\mathbb{R}^6$ with respect to ``object oriented metric" while the elapsed time gives simply the amount of the passed time for the algorithm to find the final curve.}
\label{table:results}
\end{table}
\subsection{Example}
Preserving the architectural and optimization parameters mentioned in Example\,\ref{sec:example1} and allowing the minimal singularity-free cover to take place, the result will be as depicted in Fig.\,\ref{fig:ex}. It is noteworthy that the final number of breakpoints will be 6. Finally, the geometric details of the motion is found in Table\,\ref{table:results:ex} and the related objective function per iteration descent in Fig.\,\ref{fig:graf}-right. The algorithm used to produce these results is a \textsf{Matlab} first implementation and can be subject to further improvements, especially through recoding in \textsf{C++}.
\begin{figure}[t] 
\begin{center}   
  \begin{overpic}[height=45mm]{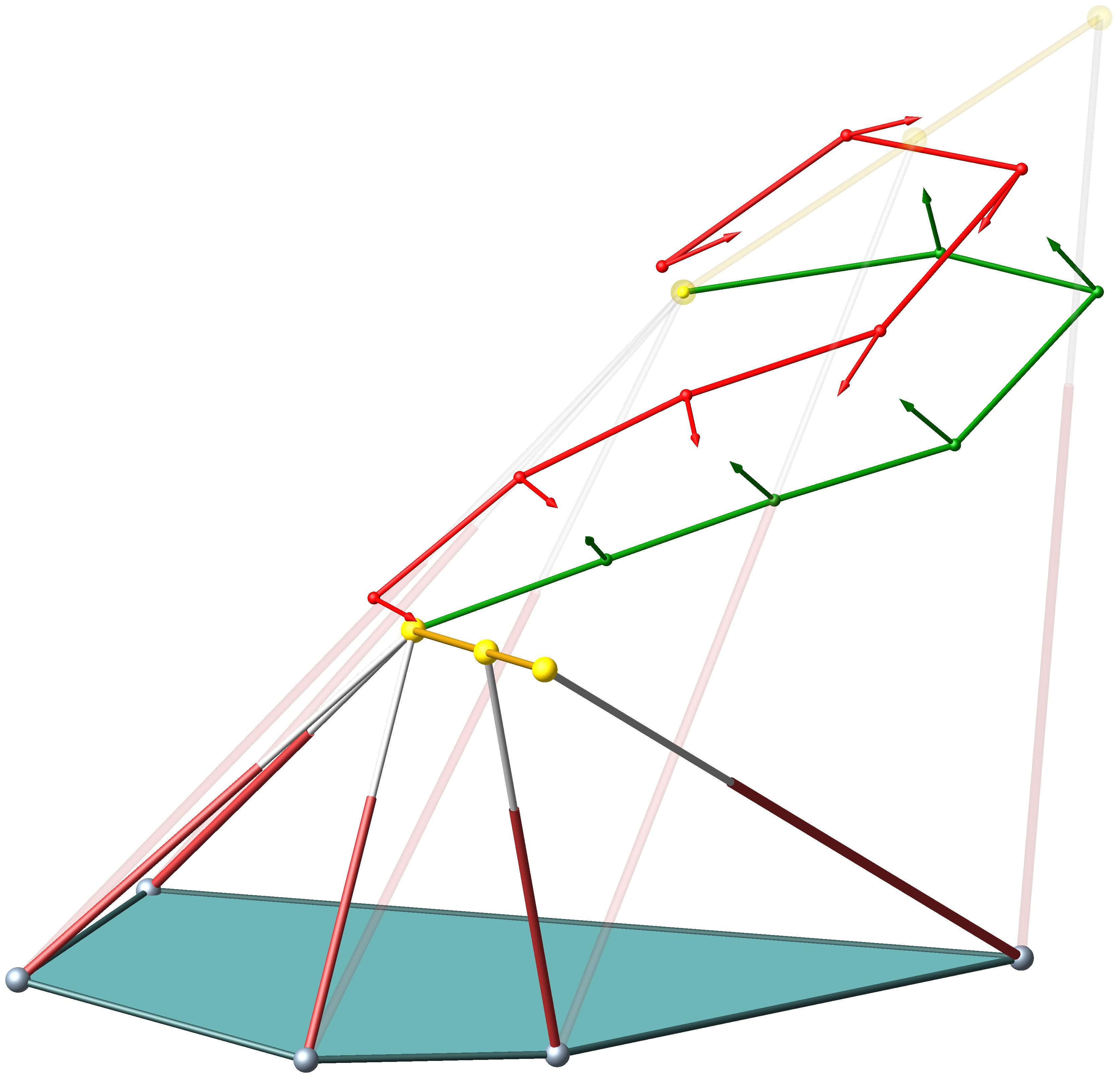}
  \end{overpic}
  \hspace*{0mm}
  \begin{overpic}[height=45mm]{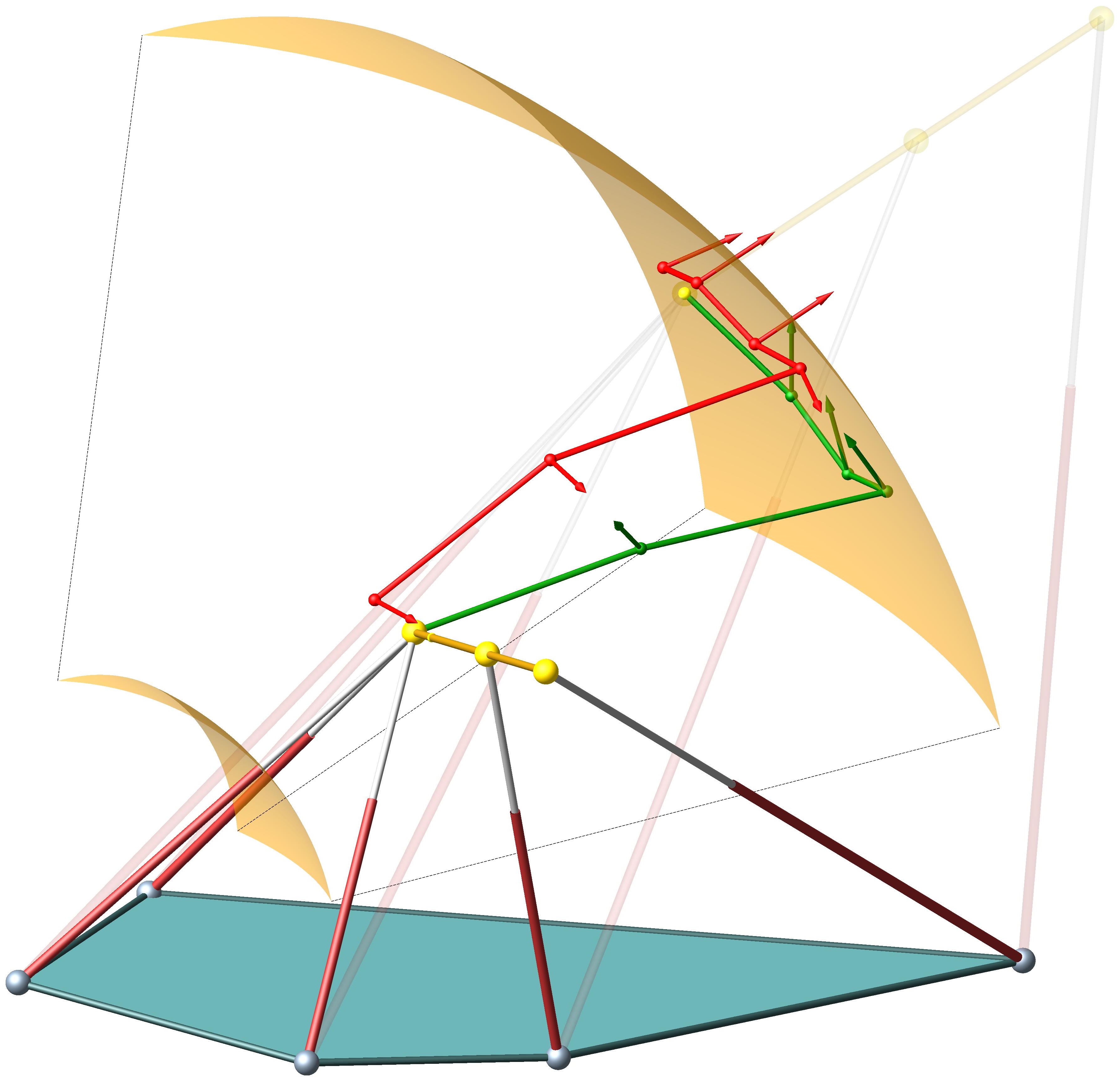}
  \end{overpic}
    \hspace*{0mm}
  \begin{overpic}[height=45mm]{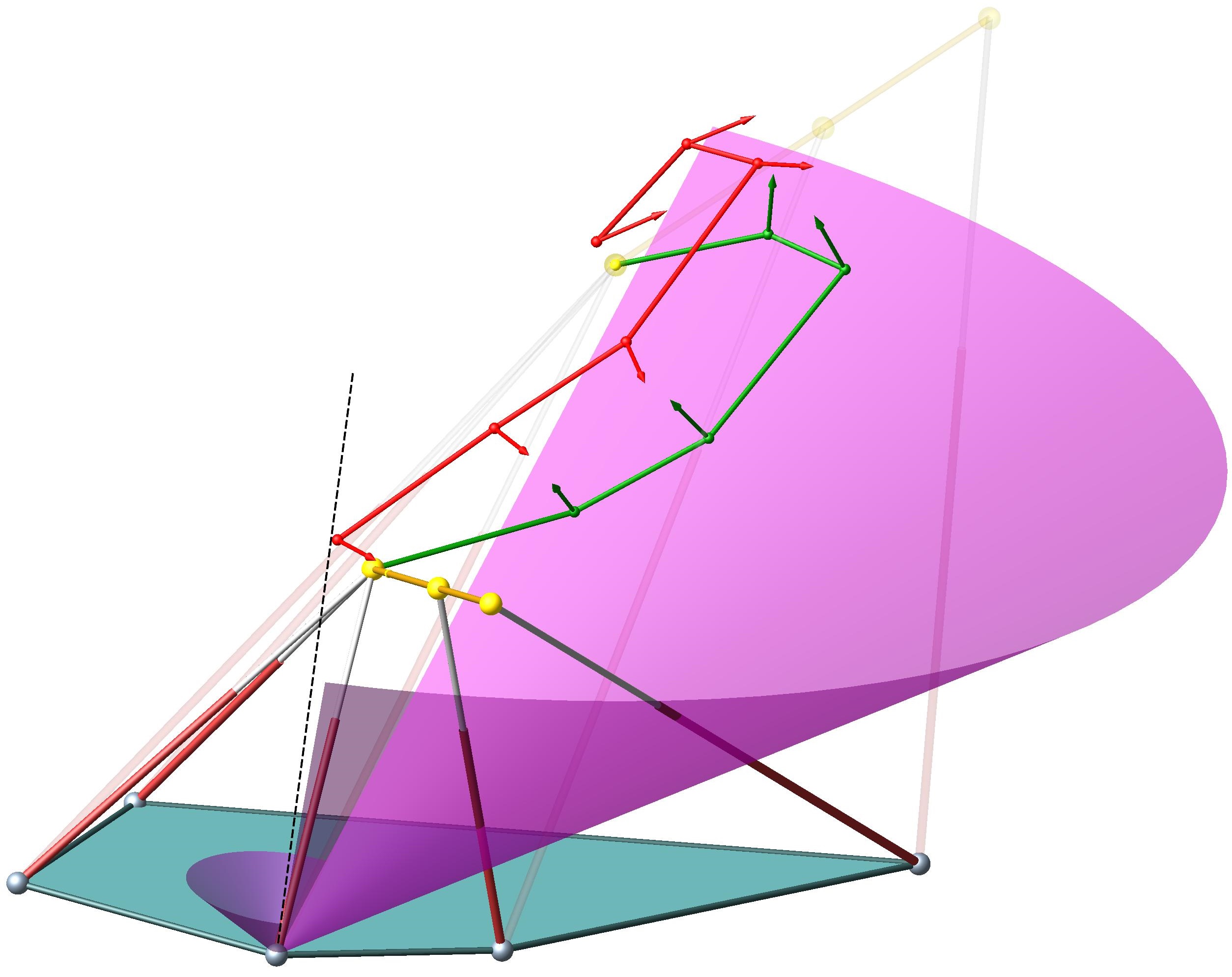}     
  \end{overpic}
  \caption{Illustration of the vector field of motions with finite singularity-free cover (cf. Section\,\ref{sec:finite:cover}). Once again the corresponding objective functions are \color{blue}blue\color{black}\ (left), \color{cyan}cyan\color{black}\ (middle) and \color{magenta}magenta\color{black}\ (right).}
	\label{fig:ex}
\end{center}
\end{figure}   
\begin{table}[h!]
\centering
\begin{small}
{
 \begin{tabular}{||c |c| c| c||} 
 \hline 
                               & Left         & Middle        & Right       \\ [0.5ex]
 \hline\hline
 Length                        & 20.6106      & 15.7346       & 16.8231     \\ [0.5ex]
 \hline
 Total curvature               & 12.9031      & 10.5039       & 11.7041     \\ [0.5ex]
 \hline
 Elapsed time                  & 2.3237 s     & 3.3512 s      & 3.7472 s    \\ [0.5ex]
 \hline
 Final number of breakpoints   & 7            & 6             & 6           \\ [0.5ex]
 \hline
\end{tabular}} \\[1ex]
\end{small}
\vspace{0.5	mm}
\caption{Length and total curvature are computed for the final curve in $\mathbb{R}^6$ with respect to ``object oriented metric", the elapsed time gives the amount of the passed time for the algorithm to find the final curve. Final number of breakpoints indicate the number of remaining breakpoints by the end of optimization process.}
\label{table:results:ex}
\end{table}
\subsection{Conclusion \& Future Research}
We demonstrated that simple pentapods imply a reduction in the number of pedal points in addition to their closed-form coordinates (cf. Section\,\ref{section:pedal}). We studied the geometric properties of simple pentapod's loci (cf. Section\,\ref{sec:algebraicgeometricsetup}).\\
Since the main purpose of the paper was to optimize a given singularity-free path between two fixed poses, we setup a cost function involving energy terms and distance to the $\Sigma$-variety (cf. Section\,\ref{sec:differentialgeometricsetup}). Obtaining the local minimum of the cost function by imploring the gradient descent method resulted in an optimized motion with the following properties: 
\vspace*{1mm}
\begin{itemize}
\item[$\bullet$] increased distance to the $\Sigma$-variety,
\item[$\bullet$] smoothness,
\end{itemize}
\vspace*{1mm}
while considering the extension/contraction limits of prismatic joints and angular limits of the base spherical joints.
Finally, the details of the algorithm along with a flowchart are given in Section\,\ref{sec:algorithmdetails} and Appendix\,\ref{appendix:C}.\\
Though the problem of feasible motions with respect to physical limits of the base spherical joints and prismatic joints yields rather simple surfaces, namely, hyperquadrics in $\mathbb{R}^6$, the problem of \emph{leg collision avoidance} and angular limits of \emph{platform spherical joints} implore a more sophisticated situation as they yield hyperquartics in $\mathbb{R}^6$. Consequently, obtaining pedals on such varieties gets more complex as even the Gr\"{o}bner basis approach does not yield a solution in the general case. One idea to ameliorate this problem is by resorting to a numerical method such as \emph{homotopy continuation method} performed by \textsf{Bertini} \cite{bates2013numerically} which will be subject to a future research. In addition, we plan to investigate the \emph{initial singularity-free path planning} problem in more detail which can be based on the special structure of the $\Sigma$-variety (cf. Theorem\,\ref{theorem:properties}).
\appendix      
\section{Proof of Theorem. \ref{theorem:properties}}
\label{appendix:A}
Since proving the theorem for one of the LP/LO cases can easily be repeated for the other case, without loss of generality, we prove the theorem for the ``LO-case":
\begin{itemize}
\item[a)] Using Lemma\,\ref{lem:LO} and basic properties of algebraic varieties (cf. \cite{cox1992ideals}) results in:
\begin{equation}
\Sigma =
\mathbf{V}(u_6)\,\cup\,\mathbf{V}(\,u_6\,(\alpha\,u_1\,+\,\beta\,u_2)\,-\,u_3\,(\,\alpha\,u_4\,+\,\beta u_5\,-\,1\,)\,) = \Sigma_1 \cup \Sigma_2.
\end{equation}
\item[b)] The set of all singular points of $\Sigma_{2}$ are the solutions of 
$
\langle \frac{\partial f}{\partial u_1} , \frac{\partial f}{\partial u_2}, \frac{\partial f}{\partial u_3}, \frac{\partial f}{\partial u_4}, \frac{\partial f}{\partial u_5}, \frac{\partial f}{\partial u_6}, f \rangle
$
where $f\in\mathbb{C}[u_1,u_2,u_3,u_4,u_5,u_6]$. If $f$ is Eq.\,\ref{LO:sing} then 
\begin{equation}
\Sigma_3 = \mathbf{V}\left(\,u_6,\,1-(\alpha u_4 + \beta u_5),\,u_3,\,\alpha\,u_1+\beta\,u_2\,\right).
\label{eq:singulareqs}
\end{equation}
Considering the following evidently smooth maps 
\begin{gather}
h_1: \mathbb{R}^5\setminus\mathbf{V}\left(\,\alpha\,u_1+\beta\,u_2\,\right)\longrightarrow \mathbb{R}^6,\\
\left(u_1,u_2,u_3,u_4,u_5\right)\longmapsto\left(u_1,u_2,u_3,u_4,u_5,\frac{u_3(\alpha u_4 + \beta u_5 - 1)}{\alpha\,u_1+\beta\,u_2}\right),  
\end{gather}
\begin{gather}
h_2: \mathbb{R}^5\setminus\mathbf{V}\left(\,1-(\alpha u_4 + \beta u_5)\,\right)\longrightarrow \mathbb{R}^6,\\
\left(u_1,u_2,u_4,u_5,u_6\right)\longmapsto\left(u_1,u_2,{\frac {u_{{6}} \left( \alpha\,u_{{1}}+\beta\,u_{{2}} \right) }{\alpha
\,u_{{4}}+\beta\,u_{{5}}-1}},u_4,u_5,u_6\right),  
\end{gather}
\begin{gather}
h_3: \mathbb{R}^5\setminus\mathbf{V}\left(\,u_3\,\right)\longrightarrow \mathbb{R}^6,\\
\left(u_1,u_2,u_3,u_5,u_6\right)\longmapsto\left(u_1,u_2,u_3,{\frac {\alpha\,u_{{1}}u_{{6}}+\beta\,u_{{2}}u_{{6}}-\beta\,u_{{3}}u_{
{5}}+u_{{3}}}{u_{{3}}\alpha}},u_5,u_6\right),  
\end{gather}
\begin{gather}
h_4: \mathbb{R}^5\setminus\mathbf{V}\left(\,u_6\,\right)\longrightarrow \mathbb{R}^6,\\
\left(u_2,u_3,u_4,u_5,u_6\right)\longmapsto\left({\frac {\alpha\,u_{{3}}u_{{4}}-\beta\,u_{{2}}u_{{6}}+\beta\,u_{{3}}u_{
{5}}-u_{{3}}}{u_{{6}}\alpha}},u_2,u_3,u_4,u_5,u_6\right),  
\end{gather}
gives 5-dimensional smooth manifolds $M_{i}:=\mathrm{graph}\left( h_{i}\right)\subset\mathbb{R}^6$ whose intersection with $\Sigma_3$ is empty. Moreover, taking an element in $\Sigma_2$ and noting the structure of $h_i$ domains implies one of the following two possibilities:
\vspace*{1mm}
\begin{itemize}
\item[$\bullet$] if it satisfies Eq.\,\ref{eq:singulareqs} then it belongs to $\Sigma_3$,
\item[$\bullet$] otherwise it belongs to at least one of the $M_i$s.
\end{itemize}
\vspace*{1mm}
Hence, by naming $M := \bigcup_{i = 1}^{4}M_{i}$ we get the required result.
\item[c)] Substituting $u_6 = 0$ in the defining equation of the hypersurface gives
\begin{equation}
\Sigma_1 \cap \Sigma_2 = \mathbf{V}\left(u_3\right) \cup \mathbf{V}(\alpha u_4 + \beta u_5 -1) = \mathcal{A} \cup \mathcal{B},
\label{proof:c}
\end{equation}
which upon observing that hypersurface equation and Eq.\,\ref{proof:c} are two constraints for pose variables or through direct computation (i.e. using \texttt{HilbertDimension} command in \textsf{Maple}) reveals that $\Sigma_1 \cap \Sigma_2$ is a union of two 4-dimensional planes.
\item[d)] Assuming $f_1$ and $f_2$ to be the defining polynomials of $\Sigma_1$ and $\Sigma_2$, to obtain the tangency one has to check the points of $\Sigma_1\cap\Sigma_2$ in which $\nabla f_1$ and $\nabla f_2$ are collinear. Doing so yields $M^{\prime}$ with the following parametrization:
\begin{gather}
\mathbb{R}^3\setminus\mathbf{V}\left(\,\alpha\,v_1+\beta\,v_2\,\right)\longrightarrow \mathbb{R}^6,\\
\left(v_1,v_2,v_3\right)\longmapsto\left(v_1, v_2, 0, \frac{1-\beta\,v_3}{\alpha}, v_3,0\right),  
\end{gather}
for the LO-case which trivially yields a 3-dimensional smooth manifold. 
\item[e)] The parametrization of $\Sigma_3$ for LO-case is:\\
\begin{equation}
\left({v_1},{v_2}\right) \longmapsto \left( -\frac{\beta\,v_1}{\alpha}, v_1, 0, \frac{1 -\beta\,v_2}{\alpha}, v_2, 0\right).
\label{LO:sigma3}
\end{equation}
Using the above parametrization (or through direct manipulation of the corresponding implicit equations) one finds that $\Sigma_3$ is a 2-dimensional plane and a subset of $\Sigma_1\cap\Sigma_2$. Additionally, by observing Eq.\,\ref{proof:c}, one deducts $\langle u_3, \left( \alpha\,u_4 + \beta\,u_5 - 1 \right)\rangle$ includes $\langle u_3\,\left( \alpha\,u_4 + \beta\,u_5 - 1 \right)\rangle$ which results in $\Sigma_3 \subset \mathcal{A} \cap \mathcal{B}$.
\end{itemize} 
\section{Flowchart}
\label{appendix:C}
For downloading the implementation in \textsf{Matlab} plus tools for plotting results in \textsf{Maple} visit \url{http://www.geometrie.tuwien.ac.at/rasoulzadeh/}. The implemented algorithm might be subject to minor updates in future releases.\\
The flowchart is modelled according to the current standard ISO 5807.
Finally, due to the fact that the variational path optimization algorithm, in its presented shape here, can be used for different optimization goals involving a path and an obstacle to avoid, the flowchart is presented as general as possible. The items exclusively related to the variational path optimization of the \emph{simple pentapods} are labelled by $\color{red}\left(\ast\right)\color{black}$. By disregarding these lines one would be able to use the same techniques for other optimization goals such as the optimization of a path with respect to a parabola in $\mathbb{R}^2$ (cf. Fig.\,\ref{fig:areal}) or similar goals in different $\mathbb{R}^n$ spaces.
\newpage
\tikzstyle{startstop} = [rectangle, rounded corners, minimum width = 3cm, minimum height = 1cm, text centered, draw = black, fill = red!30, text width = 3cm]
\tikzstyle{io} = [trapezium, trapezium left angle = 70, trapezium right angle = 110, minimum width = 3cm, minimum height = 1cm, text centered, draw = black, fill = blue!30, text width = 7.5cm]
\tikzstyle{process} = [rectangle, minimum width = 3cm, minimum height = 1cm, text centered, draw = black, fill = orange!30, text width = 7cm ]
\tikzstyle{decision} = [diamond, minimum width = 2cm, minimum height = 1cm, text centered, draw = black, fill = green!30, aspect = 2]
\tikzstyle{arrow} = [thick, ->, >=stealth]
%
\vspace*{5 pt}
\begin{tikzpicture}[node distance = 2cm] 
\node(in1)[io, text width = 7cm]{\vbox {\scriptsize
    {\begin{itemize}
        \item[1.] design parameters: $\mathsf{r}$, $\mathsf{X}$, $\mathsf{Y}$, $\mathsf{\alpha}$, $\mathsf{\beta}$ \color{red}$\left(\ast\right)$\color{black},
        \item[2.] optimization parameters: $\mathsf{initial\_ curve}$, $\mathsf{n}$, $\mathsf{\lambda}$, $\mathsf{\eta}$.
    \end{itemize}}}}; 
\node(start)[startstop, left of = in1, xshift = - 4cm]{\footnotesize start};      
\node(pro1)[process, below of = in1, text width = 10cm, yshift = -0.5 cm]{\vbox {\scriptsize
    {\begin{itemize}
        \item[1.] $\color{cyan}\mathsf{p}\color{black} := \mathsf{initial\_ curve}$, 
        \item[2.] compute $\mathsf{real\_ sorted\_ pedals}\,(\color{cyan}\mathsf{p}\color{black})$, 
        \item[3.] compute $\mathsf{all\_ real\_ distances}\,(\color{cyan}\mathsf{p}\color{black}, \mathsf{real\_sorted\_pedals})$,
        \item[4.] do necessary \textsf{inclusion} or \textsf{exclusion},
        \item[5.] \textsf{joint limit} analysis \color{red}$\left(\ast\right)$\color{black}, 
        \item[6.] compute $\mathsf{cost\_function}\,(1)$.
    \end{itemize}}}};
\node(dec1)[decision, below of = pro1, yshift = -0.8 cm]{\footnotesize \textsf{for} $\mathsf{i =1:number\_of\_iterations}$};
\node(pro2a)[process, below of = dec1, yshift = -1.5 cm, text width = 12cm]{\vbox {\scriptsize
    {\begin{itemize}
        \item[1.] solve linear system,
        \item[2.] \textsf{joint limit} analysis \color{red}$\left(\ast\right)$\color{black},
        \item[3.] compute geodesic growth: $\mathsf{geodesic\_ candidate}$, 
        \item[4.] compute bending growth: $\mathsf{bending\_ candidate}$,
        \item[5.] choose step size: $\mathsf{step\_ size = min(geodesic\_ candidate, \mathsf{bending\_ candidate})}$ ,
        \item[6.] update piecewise smooth curve: $\mathsf{up = \color{cyan}\mathsf{p}\color{black}} + \mathsf{step\_size} * \mathsf{u}$,
        \item[7.] projection into cylinder's tangents \color{red}$\left(\ast\right)$\color{black},
        \item[8.] projection into cylinder \color{red}$\left(\ast\right)$\color{black},
        \item[9.] do necessary \textsf{inclusion} or \textsf{exclusion},
        \item[10.] compute $\mathsf{cost\_ function (i + 1)}$.
    \end{itemize}}}};
\node(dec2)[decision, below of = pro2a, yshift = -1.8 cm]{\scriptsize \textsf{while} $\mathsf{cost\_function\,(i+1) > cost\_function\,(i)}$};
\node(pro2b)[process, below of = dec2, yshift = -1.5 cm, text width = 12cm]{\vbox {\scriptsize
    {\begin{itemize}
        \item[1.] $\mathsf{step\_ size = step\_ size / 2}$
        \item[2.] update piecewise smooth curve: $\mathsf{up = \color{cyan}\mathsf{p}\color{black}} + \mathsf{step\_size} * \mathsf{u}$,
        \item[3.] projection on cylinder's tangents \color{red}$\left(\ast\right)$\color{black},
        \item[4.] projection into cylinder: $\mathsf{projection\_ cylinder (up)}$ \color{red}$\left(\ast\right)$\color{black},
        \item[5.] do necessary \textsf{inclusion} or \textsf{exclusion},
        \item[6.] compute $\mathsf{cost\_ function (i + 1)}$.
    \end{itemize}}}};
\node(out1)[io, below of = pro2b, yshift = -0.2 cm, text width = 2cm]{\footnotesize $\mathsf{final\_ curve := \color{cyan}\mathsf{p}\color{black}}$};
\node(stop)[startstop, right of = out1, yshift = 0 cm, xshift = 3.75cm]{\footnotesize stop};

\draw[thick, arrow](start)--(in1);
\draw[thick, arrow](in1)--(pro1);
\draw[thick, arrow](pro1)--(dec1);
\draw[thick, arrow](dec1)--(pro2a);
\draw[thick, arrow](pro2a)--(dec2);
\draw[thick, arrow](dec2)--(pro2b);
\draw[thick, arrow](pro2b)--(out1);
\draw[thick, arrow](out1)--(stop);

\draw[thick, arrow](dec2)--node[anchor = east, xshift = -0.2cm, yshift = 0.1cm]{$\textsf{True}$}(pro2b);
\draw[thick, ->] (dec2) -- ++(7.5,0) -- ++(0,2.8) -- ++(0,4.5) -> node[xshift = 1 cm, yshift = -7.75 cm, text width=2.5cm]{$\textsf{False}$}(dec1);
\draw[thick, ->] (pro2b) -- ++(-6.5,0) -- ++(0,0) -- ++(0,3.5) -> (dec2);
\draw[thick, ->] (dec1) -- ++(-7.5,0) -- ++(0,-13) -- ++(0,0) -> (out1);

\node (-mu)[left of = in1, xshift = -0.9 cm, yshift = 0.9 cm]{\emph{Input}};
\node (-mu)[left of = pro1, xshift = -2.2 cm, yshift = 1.4 cm]{\emph{Process I}};
\node (-mu)[left of = dec1, xshift = -0.8 cm, yshift = 0.75 cm]{\emph{Decision I}};
\node (-mu)[left of = pro2a, xshift = -3.1 cm, yshift = 2.1 cm]{\emph{Process II}};
\node (-mu)[left of = dec2, xshift = -1.2 cm, yshift = 1.1 cm]{\emph{Decision II}};
\node (-mu)[left of = pro2b, xshift = -3 cm, yshift = 1.5 cm]{\emph{Process III}};
\node (-mu)[left of = out1, xshift = 0.6 cm, yshift = 0.75 cm]{\emph{Output}};
\node (-mu)[left of = dec1, xshift = -1.2 cm, yshift = -0.4 cm]{\textsf{False}};
\node (-mu)[below of = dec1, xshift = -0.8 cm, yshift = 0.6 cm]{\textsf{True}};
\end{tikzpicture}
\newpage
\bibliographystyle{siamplain}
\bibliography{references}
\end{document}